\documentclass[a4paper,12pt]{amsart}

\usepackage{amsmath}
\usepackage[latin1]{inputenc}
\usepackage{amsfonts}
\usepackage{amssymb}
\usepackage{geometry}
\usepackage{xcolor}
\newtheorem{thm}{Theorem}[section]
\newtheorem{cor}[thm]{Corollary}
\newtheorem{lem}[thm]{Lemma}

\newtheorem{defn}[thm]{Definition}
\newtheorem{rem}[thm]{Remark}

\numberwithin{equation}{section}

\def\R{\mathbb R}

\def\b{\mathrm{b}}
\def\a{a}
\def\B{B}

\newcommand{\mand}{\quad\text{and}\quad}

\title[General Reiteration Theorems for ${\mathcal R}$ and ${\mathcal L}$ classes]
{General Reiteration Theorems for ${\mathcal R}$ and ${\mathcal L}$ classes:
case of right ${\mathcal R}$-spaces and left ${\mathcal L}$-spaces}

\author[Fern\'andez-Mart\'{\i}nez]{Pedro Fern\'andez-Mart\'{\i}nez}
\address[Pedro Fern\'andez-Mart\'{\i}nez and Teresa M. Signes]{Departamento de Matem\'aticas \\
Facultad de Matemáticas \\ Universidad de Murcia \\ Campus de
Espinardo \\ 30071 Espinardo (Murcia), Spain}

\author[Signes]{Teresa M. Signes}

\date{\today}

\begin{document}
\begin{abstract}
Given $E_0, E_1, E, F$ rearrangement invariant spaces, $\a, \b, \b_0, \b_1$ slowly varying functions and $0\leq \theta_0<\theta_1\leq 1$, we characterize the interpolation space
$$(\overline{X}_{\theta_0,\b_0,E_0}, \overline{X}^{\mathcal R}_{\theta_1, \b_1,E_1,\a,F})_{\theta,\b,E}\mand
(\overline{X}^{\mathcal L}_{\theta_0, \b_0,E_0,\a,F}, \overline{X}_{\theta_1,\b_1,E_1})_{\theta,\b,E},$$
for all possible values of $\theta\in[0,1]$. Applications to interpolation identities for grand and small Lebesgue spaces, Gamma spaces and $A$ and $B$-type spaces are given.
\end{abstract}
\maketitle

\section{Introduction}\label{introduction}

Reiteration theorems are important results in Interpolation Theory. These results not only assure the interpolation process is stable under reiteration, but also they are very useful identifying interpolation spaces.
The classical results can be found in  the monographs \cite{Bennett-Sharpley,Bergh-Lofstrom,B-K,triebel}. Additionally, there is an extensive literature concerning explicit reiteration formulae in various special cases, see e.g. \cite{AEEK,Ah1,ALM,CSe-2,Do,EOP,GOT,Mi,Nill}.

This paper is the second of a series in which we study reiteration results for couples formed by the spaces
\begin{equation}\label{a0}
\overline{X}_{\theta,\b,E},\  \overline{X}^{\mathcal R}_{\theta,\b,E,\a,F},\  \overline{X}^{\mathcal L}_{\theta,\b,E,\a,F}.  
\end{equation}
Here, $0 \leq \theta \leq 1$, $a$ and $b$ are slowly varying functions and, finally, $E$ and $F$ are rearrangement invariant function  (r.i.)  spaces. In fact,   given a couple  $\overline{X}=(X_0,X_1)$,  these spaces are defined as 
$$
 \overline{X}_{\theta,\b,E}=\Big\{f\in X_0+X_1\;\colon\;
\big \| t^{-\theta} {\b}(t) K(t,f) \big \|_{\widetilde{E}} < \infty\Big\},
$$
$$\overline{X}_{\theta,\b,E,a,F}^{\mathcal L}=\{f\in X_0+X_1\;\colon\;
\Big\|\b(t) \|s^{-\theta} a(s) K(s,f) \|_{\widetilde{F}(0,t)}\Big\|_{\widetilde{E}} < \infty\}$$
and 
$$\overline{X}_{\theta,\b,E,a,F}^{\mathcal R}=\{f\in X_0+X_1\;\colon\; \Big \|  \b(t) \|   s^{-\theta} a(s) K(s,f) \|_{\widetilde{F}(t,\infty)}      \Big   \|_{\widetilde{E}}<\infty\}.$$

\ 

In the previous paper \cite{FMS-RL1}, we identified the interpolation spaces
\begin{equation}\label{a1}
\Big(\overline{X}^{\mathcal R}_{\theta_0, \b_0,E_0,\a,F}, \overline{X}_{\theta_1,\b_1,E_1}\Big)_{\theta,\b,E}\mand 
\Big(\overline{X}_{\theta_0,\b_0,E_0}, \overline{X}^{\mathcal L}_{\theta_1, \b_1,E_1,\a,F}\Big)_{\theta,\b,E}.
\end{equation} 
This time we shall focus in the ``dual'' situation
\begin{equation}\label{a2}
\Big(\overline{X}_{\theta_0,\b_0,E_0}, \overline{X}^{\mathcal R}_{\theta_1, \b_1,E_1,\a,F}\Big)_{\theta,\b,E}\mand
\Big(\overline{X}^{\mathcal L}_{\theta_0, \b_0,E_0,\a,F}, \overline{X}_{\theta_1,\b_1,E_1}\Big)_{\theta,\b,E},
\end{equation}
when $0\leq \theta_0<\theta_1\leq 1$, $0\leq \theta\leq 1$, $\a, \b, \b_0, \b_1$  are slowly varying functions and $E_0, E_1, E, F$ are r.i.\!  spaces.  

\

We remark that the identities in \eqref{a2} do not follow from those in \eqref{a1} by a usual  ``symmetry'' argument
 (i.e., interchanging $X_0$ and $X_1$), since such argument would not preserve the condition $0\leq \theta_0<\theta_1\leq1$, which is crucial in  the identification of the spaces  in \eqref{a1}. 
Moreover, in the limiting cases $\theta=0,\ 1$ the reiteration spaces obtained in \eqref{a2} will no longer belong to the scales in \eqref{a0}.
In fact, new interpolation functors
\begin{equation}\label{a3}
\overline{X}^{\mathcal R,\mathcal R}_{\theta,c,E,\b,E,\a,G}\mand  \overline{X}^{\mathcal L,\mathcal L}_{\theta,c,E,\b,F,\a,G}
\end{equation}
will be needed; see Definition \ref{defLRR}.

\ 

As in \cite{FMS-RL1}, a motivation  for this study arises from various recent applications (see \cite{AFH,AFFGR,FFGKR})
to the so-called
 \textit{grand} and \textit{small Lebesgue} spaces \[
L^{p),\alpha}\mand L^{(p,\alpha}, \quad \mbox{$1<p<\infty$, \;$\alpha>0$;}
\]
see Definition \ref{def_gLp} below.  As
observed in \cite{FK, op1}, one can write
\[
L^{p),\alpha}=(L_1,L_\infty)^{\mathcal R}_{1-\frac1p,\ell^{-\frac{\alpha}{p}}(t),L_\infty,1,L_p} \mand
L^{(p,\alpha}=(L_1,L_\infty)^{\mathcal L}_{1-\frac1p,\ell^{-\frac{\alpha}{p}+\alpha-1}(t),L_1,1,L_p}\;,
\]
with $\ell(t)=1+|\log(t)|$, $t\in(0,1)$. 
So, \eqref{a2} allows in particular to identify 
the  interpolation spaces
$$\big(L_{p_0},L^{p_1),\alpha}\big)_{\theta,\b,E}\quad\mbox{and}\quad \big(L^{(p_0,\alpha},L_{p_1}\big)_{\theta,\b,E}.$$
for $1\leq p_0<p_1\leq \infty$, $\alpha>0$, and $0\leq \theta\leq 1$. Moreover, using additionally reiteration results from \cite{FMS-2} or using limiting cases for $\theta_0$ and $\theta_1$ in \eqref{a2}, one computes as well the pairs 
$$\big(L^{(p_0,\alpha},L^{p_1),\beta}\big)_{\theta,\b,E},\quad \big(L\log L,L^{p_1),\beta}\big)_{\theta,\b,E}\mand \big(L^{(p_0,\alpha},L_{\exp}\big)_{\theta,\b,E},$$ 
see Theorem \ref{thm58} and Corollaries \ref{714}, \ref{712} below.

\

Some of these special cases are contained in the recent papers \cite{AFH,AFFGR,FFGKR},
together with other interpolation formulae for pairs involving grand or small Lebesgue (Lorentz) spaces. 
Our goal here is to present a unified study for such identities
in the setting of general couples $\overline{X}=(X_0,X_1)$ of (quasi-) Banach spaces, 
and arbitrary parameters $\b$ and $E$. 
This point of view, besides being more general, also produces new formulas compared to \cite{AFH,AFFGR,FFGKR},
and allows to apply the results to  other situations, such as \textit{Gamma} spaces and $A$ and $B$-type spaces; see \S\ref{applications} below.

\

As in \cite{FMS-RL1}, the proofs use a direct approach, which closely follows the classical methods of reiteration. 
The main point is to obtain Holmstedt type formulae for the interpolation couples described above; 
these can  be seen as quantitative forms of the reiteration theorems. 
We also make use of standard techniques that already appeared in \cite{FMS-RL1}, such as Hardy type inequalities in the context of 
r.i. spaces and slowly varying functions (see \S\ref{lemmas}),
and an estimate that is specific of this situation (Lemma \ref{thmFMS-4}). 

\

The paper is organized as follows. In Section 2 we recall basic concepts regarding rearrangement invariant spaces and  slowly varying functions.  We also describe the  interpolation methods we shall work with, namely $\overline{X}_{\theta,\b,E}$, the ${\mathcal R}$ and ${\mathcal L}$spaces, $\overline{X}^{\mathcal R}_{\theta,\b,E,a,F}$, $\overline{X}^{\mathcal L}_{\theta,\b,E,a,F}$, and the new constructions in \eqref{a3}.
Generalized Holmstedt type formulae for the $K$-functional  of the couples involved  can be found in Section 3.
The reiteration results appear in Section 4 and finally Section 5 is devoted to applications.

\section{Preliminaries} \label{preliminaries}

We refer to the monographs \cite{Bennett-Sharpley,Bergh-Lofstrom,B-K,KPS,triebel} for the  basic concepts  and facts on Interpolation Theory and Banach function spaces. A Banach function space $E$ on $(0,\infty)$ is called \textit{rearrangement invariant}  (r.i.) if, for any two measurable functions $f$, $g$,
\begin{equation*}
g\in E \text{ and } f^*\leq g^*  \Longrightarrow f\in E \text{ and } \|f\|_E\leq \|g\|_E,
\end{equation*}
where $f^*$ and $g^*$ stand for the decreasing rearrangements of $f$ and $g$.
Following  \cite{Bennett-Sharpley},  we assume that every Banach function space $E$ enjoys the \textit{Fatou property}. Under this assumption  every r.i. space
$E$ can be obtained by applying an exact interpolation method
to the couple $(L_{1}, L_{\infty})$.

Along this paper we will handle two different measures on $(0,\infty)$; the usual Lebesgue measure and the homogeneous measure $\nu(A)=\int _0^\infty
\chi_A(t)\tfrac{dt}{t}$.
We use a tilde to denote rearrangement invariant spaces with respect to the second measure.
For example
$$  \|f\|_{\widetilde{L}_{1}} = \int_{0}^{\infty}  |f(t)|  \frac{dt}{t}\quad\mbox{and}\quad \|f\|_{\widetilde{L}_\infty}=\|f\|_{L_\infty}.$$
If  the space $E$ is obtained by applying the interpolation functor $\mathcal{F}$ on the couple $(L_{1}, L_{\infty})$,  $E=\mathcal{F}(L_1,L_\infty)$,
then we write $\widetilde{E}=\mathcal{F}(\widetilde{L}_{1}, L_{\infty})$.

Sometimes we will need to restrict the space to some partial interval $(a,b)\subset (0,\infty)$.
Then we will use the notation $E(a,b)$ and $\widetilde{E}(a,b)$. The norms in $E$ and $E(a,b)$ are related by: $\|f\|_{E(a,b)}=\|f(t)\chi_{(a,b)}(t)\|_{E}$.

For  two (quasi-) Banach spaces  $X$ and $Y$, we write $X\hookrightarrow Y$ if $Y\subset X$ and the natural embedding is continuous. The symbol  $X=Y$ means that $X\hookrightarrow Y$ and $Y\hookrightarrow X$.

Let $A$ and $B$ be two non-negative quantities depending on certain parameters.  We write $A\lesssim B$ if there is a constant $c>0$, independent of the parameters involved in $A$ and $B$, such that  $A\leq cB$.  If $A\lesssim B$ and $B\lesssim A$,  we say that $A$ and $B$ are equivalent and write $A\sim B$.

\subsection{Slowly varying functions}\hspace{1mm}

\vspace{2mm}
In this subsection we recall the definition and basic properties of \textit{slowly varying functions}.  See  \cite{BGT,Matu}.

\begin{defn}\label{def1}
A positive Lebesgue measurable function $\b$, $0\not\equiv\b\not\equiv\infty$,
is said to be \textit{slowly varying} on $(0,\infty)$ (notation $\b\in SV$) if, for each $\varepsilon>0$, the function $t \leadsto t^\varepsilon\b(t)$  is equi\-va\-lent to a non-decreasing function on $(0,\infty)$
and $ t \leadsto t^{-\varepsilon}\b(t)$ is equivalent to a non-increasing function on $(0,\infty)$.
\end{defn}

Examples of $SV$-functions include powers of logarithms,
$$\ell^\alpha(t)=(1+|\log t|)^\alpha,\quad t>0,\quad \alpha\in\R,$$ ``broken" logarithmic functions defined as
\begin{equation}\label{ellA}
\ell^{(\alpha, \beta)}(t)=\left\{\begin{array}{ll}
\ell^\alpha(t),& 0<t\leq1 \\
\ell^\beta(t)
,& t>1
\end{array}\right.,\quad (\alpha, \beta)\in\R^2,
\end{equation}
reiterated logarithms $(\ell\circ\ldots\circ\ell)^\alpha (t),\ \alpha\in\R,\ t>0$ and also the family of functions  $\exp(|\log t|^\alpha), \ \alpha \in (0,1), \; t >0$. 

In the following lemmas we summarized some of the basic properties of slowly varying functions.

\begin{lem}\label{lem0}
Let $\b, \b_1, \b_2\in SV$.
\begin{itemize}
\item[(i)]  Then $\b_1\b_2\in SV$, $\b(1/t)\in SV$ and $\b^r\in SV$ for all $r\in\R$.
\item[(ii)] If $\alpha>0$, then $\b(t^\alpha \b_1(t))\in SV$.
\item[(iii)] If $\epsilon,s>0$ then there are positive constants $c_\epsilon$ and $C_\epsilon$ such that
$$c_\epsilon\min\{s^{-\epsilon},s^\epsilon\}\b(t)\leq \b(st)\leq C_\epsilon \max\{s^\epsilon,s^{-\epsilon}\}\b(t)\quad \mbox{for every}\  t>0.$$
\end{itemize}
\end{lem}

\vspace{2mm}
\begin{lem} \label{lem1}
Let $E$ be an r.i. space on $(0,\infty)$ and $\b\in SV$.
\begin{itemize}
\item[(i)] If $\alpha>0$, then, for all $t>0$,
$$\|s^\alpha\b(s)\|_{\widetilde{E}(0,t)}\sim t^\alpha\b(t)\quad \mbox{ and } \quad
\|s^{-\alpha}\b(s)\|_{\widetilde{E}(t,\infty)}\sim t^{-\alpha}\b(t).$$
\item[(ii)] The following functions belong to $SV$
$$\B_0(t):=\|\b\|_{\widetilde{E}(0,t)}
\quad \text{and}\quad
\B_\infty(t):=\|\b\|_{\widetilde{E}(t,\infty)}, \quad t >0.$$
\item[(iii)] For all $t >0$,
$$\b(t) \lesssim \| \b \|_{\widetilde{E}(0,t)} \quad \text{and} \quad
\b(t) \lesssim \| \b \|_{\widetilde{E}(t,\infty)}.$$
\end{itemize}
\end{lem}

We refer to \cite{GOT,FMS-1} for the proof of Lemma \ref{lem0} and \ref{lem1}, respectively.
\begin{rem}\label{rem23}
\textup{ The property (iii) of Lemma \ref{lem0} implies that if $\b\in SV$ is such that $\b(t_0)=0$ ($\b(t_0)=\infty$) for some $t_0>0$, then $\b\equiv0$ ($\b\equiv\infty$). Thus, by Lemma \ref{lem1} (ii), if $\|\b\|_{\widetilde{E}(0,1)}<\infty$ then
$\|\b\|_{\widetilde{E}(0,t)}<\infty$ for all $t>0$, and if $\|\b\|_{\widetilde{E}(1,\infty)}<\infty$ then
$\|\b\|_{\widetilde{E}(t,\infty)}<\infty$ for all $t>0$.}

\textup{Moreover, if $f\sim g$ then, using Definition \ref{def1} and Lemma \ref{lem0} (iii), one can show that  $\b\circ f\sim \b\circ g$ for any $b\in SV$.}
\end{rem}

\subsection{Interpolation Methods} \label{interpolation methods}\hspace{2mm}

\vspace{2mm}
In what follows $\overline{X}=(X_{0}, X_{1})$ will be a \textit{compatible (quasi-) Banach couple}, that is, two  (quasi-) Banach spaces continuously embedded in a Hausdorff topological vector space. The \textit{Peetre $K$-functional} $K(t,f)\equiv K(t,f;X_0,X_1)$ is defined for $f\in X_0+X_1$ and $t>0$ by
 \begin{align*}
K(t,f;X_0,X_1)=\inf \Big \{\|f_0\|_{X_0}+t\|f_1\|_{X_1}:\ f=f_0+f_1,\ f_i\in X_i , \; i=0,1
\Big \}.
\end{align*}
It is known that the function $t\rightsquigarrow K(t,f)$ is non-decreasing, while $t\rightsquigarrow t^{-1}K(t,f)$, $t>0$, is non-increasing.  Other important property of the $K$-functional is the fact that 
\begin{equation}\label{eKK}
K(t,f;X_0,X_1)=tK(t^{-1},f;X_1,X_0)\quad \mbox{for all}\ t>0,
\end{equation}
(see \cite[Chap. 5, Proposition 1.2]{Bennett-Sharpley}).

Now, we recall the definition and some properties of the real interpolation method  $\overline{X}_{\theta,\b,E}$ and of the limiting $\mathcal{L}$ and $\mathcal{R}$ constructions. See \cite{FMS-1}  for the proof of the results of this subsection.

\begin{defn}\label{defrealmethod}
Let $E$ be an r.i.  space,  $\b\in SV$ and $0\leq\theta \leq 1$. The real interpolation space $\overline{X}_{\theta,\b,E}\equiv(X_0,X_1)_{\theta, \b, E}$ consists of all $f$ in $X_{0} + X_{1}$ that satisfy
$$ \|f\|_{\theta,\b,E} := \big \| t^{-\theta} {\b}(t) K(t,f) \big \|_{\widetilde{E}} < \infty.$$
\end{defn}

The space  $\overline{X}_{\theta, \b,E}$ is a (quasi-) Banach space, and is an intermediate space for the couple  $(X_0, X_1)$, that is,
$$X_0\cap X_1\hookrightarrow \overline{X}_{\theta, \b,E}\hookrightarrow X_0+X_1,$$
provided that one of the following conditions holds
$$ \left\{\begin{array}{cc}
\quad 0< \theta< 1&\\
\theta=0, & \|\b\|_{\widetilde{E}(1,\infty)}<\infty\\
\theta=1, & \|\b\|_{\widetilde{E}(0,1)}<\infty.
\end{array}\right.
$$
If none of the above conditions holds, then $\overline{X}_{\theta, \b,E}=\{0\}$.

When $E=L_q$ and $\b \equiv 1$, then $\overline{X}_{\theta,\b,E}$ coincides with the classical real interpolation space $\overline{X}_{\theta,q}$. We emphasize, however, that the spaces $\overline{X}_{\theta,b,E}$ are well defined even for the extremal values of the parameter $\theta =0$ and $\theta =1$. Thus, this scale contains extrapolation spaces in the sense of Milman~\cite{Mi} and Gómez and Milman \cite{Go-Mi}.  The interpolation spaces $\overline{X}_{\theta,\b,L_q}$ have been studied in detail by Gogatishvilli, Opic and Trebels in~\cite{GOT}, while the special cases $\overline{X}_{\theta,\ell^{(\alpha,\beta)}(t),L_q}$, see \eqref{ellA}, were considered earlier by  Evans, Opic and Pick in \cite{EOP} and by Evans and Opic in \cite{EO}. See also \cite{clku,Do,gustav,per,tesisalba}, among other articles.

In the next remark we collect some elementary estimates  that will be used in the rest of the paper.
\begin{rem}
\textup{ Using Lemma \ref{lem0} (iv) and the monotonicity of $t \leadsto K(t,\cdot)/t$,  it is easy to check that for $t>0$ and $f\in X_0+X_1$ 
\begin{equation}\label{e1}
t^{-\theta}\b(t)K(t,f)\lesssim\big \| s^{-\theta} {\b}(s) K(s,f) \big \|_{\widetilde{E}(0,t)},\qquad 0\leq \theta<1 \end{equation}
and
\begin{equation}\label{e2}
t^{-1}\|\b\|_{\widetilde{E}(0,t)}K(t,f)\lesssim\big \| s^{-1} {\b}(s) K(s,f) \big \|_{\widetilde{E}(0,t)}.\qquad\qquad\qquad\qquad\qquad
\end{equation}
Similarly, 
\begin{equation}\label{e3}
t^{-\theta}\b(t)K(t,f)\lesssim\big \| s^{-\theta} {\b}(s) K(s,f) \big \|_{\widetilde{E}(t,\infty)},\qquad 0<\theta\leq 1
\end{equation}
and for $\theta=0$
\begin{equation}\label{e4}
\|\b\|_{\widetilde{E}(t,\infty)}K(t,f)\lesssim\big \|{\b}(s) K(s,f) \big \|_{\widetilde{E}(t,\infty)}.\qquad\qquad\qquad\qquad
\end{equation}
It is worth to remark that by \eqref{e2} and \eqref{e3}, we have 
\begin{equation}\label{eK}
K(t,f)\lesssim \frac{t^\theta}{\b(t)}\|f\|_{\theta,\b,E},\qquad 0<\theta<1,
\end{equation}
for all $t>0$ and  $f\in \overline{X}_{\theta,\b,E}$. In the cases $\theta=0,1$ the estimate \eqref{eK} is also true if we replace $\b(t)$ by $\|\b\|_{\widetilde{E}(t,\infty)}$ or $\|\b\|_{\widetilde{E}(0,t)}$, respectively.}
\end{rem}

\begin{defn}\label{defLR}
 Let $E$, $F$ be two r.i. spaces, $a, \b\in SV$ and $0\leq \theta\leq 1$. The space $\overline{X}_{\theta,\b,E,a,F}^{\mathcal L}\equiv(X_0,X_1)_{ \theta,\b,E,a,F}^{\mathcal L}$ consists of all $f \in X_{0} + X_{1}$ for which 
$$ \|f \|_{\mathcal L;\theta,\b,E,a,F}:=
\Big\|\b(t) \|s^{-\theta} a(s) K(s,f) \|_{\widetilde{F}(0,t)}\Big\|_{\widetilde{E}} < \infty . $$
\end{defn}
This is  a (quasi-) Banach  intermediate space for the couple $\overline{X}$, 
$$ X_0\cap X_1 \hookrightarrow  \overline{X}_{\theta,\b,E,a,F}^{\mathcal L} \hookrightarrow X_0+X_1,$$
provided that 
$$\left\{\begin{array}{cl}
0 < \theta < 1& \|\b\|_{\widetilde{E}(1,\infty)} <\infty \\
\theta =0& \|\b\|_{\widetilde{E}(1,\infty)} \!<\infty,\quad \Big \| \b(t)\|a\|_{\widetilde{F}(1,t)}  \Big \|_{\widetilde{E}(1,\infty)} \!< \infty, \|\a\b\|_{\widetilde{E}(1,\infty)}<\infty\\
\theta =1& \|\b\|_{\widetilde{E}(1,\infty)} \!<\infty,\quad  \Big \| \b(t) \|a\|_{\widetilde{F}(0,t)} \Big \|_{\widetilde{E}(0,1)} \!< \infty.
\end{array}\right.
$$
If none of these conditions holds, then $\overline{X}^{\mathcal L}_{\theta, \b,E,a,F}=\{0\}$.

\begin{defn}
 Let $E$, $F$ be two r.i. spaces, $a, \b\in SV$ and $0\leq \theta\leq 1$. The space
$\overline{X}_{\theta,\b,E,a,F}^{\mathcal R}\equiv(X_0,X_1)_{\theta,\b,E,a,F}^{\mathcal R}$ consists of all $f\in X_0+X_1$ for which
$$\| f  \|_{\mathcal{R};\theta,\b,E,a,F} := \Big \|  \b(t) \|   s^{-\theta} a(s) K(s,f) \|_{\widetilde{F}(t,\infty)}      \Big   \|_{\widetilde{E}}<\infty.$$
\end{defn}

The space $\mathcal{R}$ is a (quasi)-Banach  intermediate space for the couple $\overline{X}$, that is,
$$ X_0\cap X_1 \hookrightarrow  \overline{X}_{\theta,\b,E,a,F}^{\mathcal R} \hookrightarrow X_0+X_1$$
provided that
$$\left\{\begin{array}{cl}
0 < \theta < 1& \|\b\|_{\widetilde{E}(0,1)} <\infty \\
\theta =0& \|\b\|_{\widetilde{E}(0,1)} \!<\infty,\quad \Big \| \b(t)\|a\|_{\widetilde{F}(t,\infty)}  \Big \|_{\widetilde{E}(1,\infty)} \!< \infty\\
\theta =1& \|\b\|_{\widetilde{E}(0,1)} \!<\infty,\quad  \Big \| \b(t) \|a\|_{\widetilde{F}(t,1)} \Big \|_{\widetilde{E}(0,1)} \!< \infty, \|\a\b\|_{\widetilde{E}(0,1)}<\infty.
\end{array}\right.
$$
If none of these conditions holds, then $\overline{X}^{\mathcal L}_{\theta, \b,E,a,F}$ is a trivial space.

The above definitions  generalize a previous notion by  Evans and Opic in \cite{EO}, in which $\a, \b$ are broken logarithms while $E, F$ remain within the classes $L_q$. Earlier versions of these spaces appeared in a paper  by Doktorskii \cite{Do}, with $a, \b$ powers of logarithms and  $(X_0,X_1)$ an ordered couple.  These spaces  also appear in  the work of  Gogatishvili, Opic and Trebels \cite{GOT} and Ahmed \textit{et al.} \cite{AEEK}. In all these cases the spaces $E$ and $F$ remain within the $L_q$ classes.

Some inclusions between the three constructions can be obtained 
using  inequalities  \eqref{e1}-\eqref{e4}. In fact,
\begin{equation}\label{ec18}
\begin{aligned}
 &\overline{X}_{\theta,\b,E,a,F}^{\mathcal L}  \hookrightarrow\overline{X}_{\theta,\b\a,E} & \text{ if }\ \ 0\leq \theta< 1 \\
 &\overline{X}_{1,\b,E,a,F}^{\mathcal L}  \hookrightarrow\overline{X}_{1,\b\|\a\|_{\widetilde{F}(0,t)},E} & \text{ if } \theta= 1,
\end{aligned}
\end{equation}
and
\begin{equation}\label{ec19}
\begin{aligned}
& \overline{X}_{\theta,\b,E,a,F}^{\mathcal R}  \hookrightarrow\overline{X}_{\theta,\b\a,E} & \text{ if } 0< \theta \leq 1 \\
&\overline{X}_{0,\b,E,a,F}^{\mathcal R}  \hookrightarrow\overline{X}_{0,\b\|\a\|_{\widetilde{F}(t,\infty)},E} & \text{ if } \theta =0.
\end{aligned}
\end{equation}
See \cite[Lemma 6.7]{FMS-1} for the identities in the  case $E=F=L_q$.

\

The following lemma shows the effects that changing the order of the spaces in the couple have on these  constructions.

\begin{lem}\label{symLR}
Let $E, F$ be r.i. spaces, $\a$, $\b\in SV$ and $0\leq\theta\leq1$. Then
$$(X_0,X_1)_{\theta,\b,E}=(X_1,X_0)_{1-\theta,\overline{\b},E},\quad (X_0,X_1)^{\mathcal{L}}_{\theta,\b,E,a,F}=(X_1,X_0)^{\mathcal{R}}_{1-\theta,\overline{\b},E,\overline{a},F},$$
where $\overline{a}(t)=a(1/t)$ and  $\overline{\b}(t)=\b(1/t)$,  $t>0$.
\end{lem}

We conclude this subsection with some inequalities that will be used later.

\begin{lem}\label{lemLRK}
Let $E, F$ be  r.i. spaces, $a, \b\in SV$ and $0\leq \theta\leq 1$. Then, for all $f\in X_0+X_1$ and $u>0$
\begin{equation}\label{e5}
u^{-\theta}\a(u)\|\b\|_{\widetilde{E}(0,u)}K(u,f)\lesssim \Big\| \b(t) \|s^{- \theta} \a(s) K(s,f)\|_{ \widetilde{F}(t,u)}\Big\|_{ \widetilde{E}(0,u)}
\end{equation}
and
\begin{equation}\label{e6}
u^{-\theta}\a(u)\b(u)K(u,f)\lesssim \Big\| \b(t) \|s^{- \theta} \a(s) K(s,f)\|_{ \widetilde{F}(t,\infty)}\Big\|_{ \widetilde{E}(u,\infty)}.
\end{equation}
\end{lem}

\begin{proof}
We refer to \cite[Lemma 2.12]{FMS-RL1} for the proof of  the first inequality \eqref{e5}.
The second  inequality \eqref{e6} is an easy consequence of Lemma \ref{lem1} (i) and the monotonicity of the $K$-functional. Indeed,
\begin{align*}
\Big\| \b(t) \|s^{-\theta} \a(s) K(s,f)\|_{ \widetilde{F}(t,\infty)}\Big\|_{ \widetilde{E}(u,\infty)}& \gtrsim
K(u,f)\Big\| \b(t) \|s^{- \theta} \a(s)\|_{ \widetilde{F}(t,\infty)}\Big\|_{ \widetilde{E}(u,\infty)}\\
&\sim K(u,f)\|t^{-\theta} \b(t) \a(t)\|_{ \widetilde{E}(u,\infty)}\\&\sim u^{-\theta}\a(u)\b(u)K(u,f).
\end{align*}
\end{proof}

Consequently, when $0\leq\theta\leq1$ and $f\in\overline{X}^{\mathcal R}_{\theta,\b,E,\a,F}$, we have
\begin{equation}\label{eK2}
K(u,f)\lesssim \frac{u^\theta}{\a(u)\|\b\|_{\widetilde{E}(0,u)}}\|f\|^{\mathcal R}_{\theta,\b,E,a,F},\quad u>0.
\end{equation}

\subsection{New extremal spaces}\label{extremalspaces}\hspace{1mm}

\vspace{2mm}
Our next definition introduces two new types of extremal interpolation spaces that will appear in relation with the extreme reiteration results that will be studied in \S 4. 
\begin{defn}\label{defLRR}
Let $E, F, G$ be  r.i. spaces, $\a, \b, c \in SV$ and $0\leq \theta\leq1$.  The space
$\overline{X}_{\theta,c,E,\b,F,\a,G}^{\mathcal L,\mathcal L}\equiv(X_0,X_1)_{ \theta,c,E,\b,F\a,G}^{\mathcal L,\mathcal L}$ is the set of all $f\in X_0+X_1$ for which 
$$ \|f \|_{\mathcal L,\mathcal L;\theta,c,E,\b,F,\a,G} :=
\bigg\|c(u)\Big\|\b(t) \|s^{-\theta} \a(s) K(s,f) \|_{\widetilde{G}(0,t)}\Big\|_{\widetilde{F}(0,u)}\bigg\|_{\widetilde{E}}<\infty.$$

Similarly, the space
$\overline{X}_{\theta,c,E,\b,F,\a,G}^{\mathcal R,\mathcal R}\equiv(X_0,X_1)_{ \theta,c,E,\b,F\a,G}^{\mathcal R,\mathcal R}$ is the set of all  $f\in X_0+X_1$ such that
$$ \|f \|_{\mathcal R,\mathcal R;\theta,c,E,\b,F,\a,G} :=
\bigg\|c(u)\Big\|\b(t) \|s^{-\theta} \a(s) K(s,f) \|_{\widetilde{G}(t,\infty)}\Big\|_{\widetilde{F}(u,\infty)}\bigg\|_{\widetilde{E}} < \infty.$$
\end{defn}

A standard reasoning (see \cite{Bennett-Sharpley,EO}) shows that the above two classes are (quasi-) Banach spaces, provided $X_0$ and $X_1$ are so.  Moreover,  using estimates \eqref{e1}-\eqref{e4}  one can obtain the following inclusions

$$\begin{aligned}
 &\overline{X}^{\mathcal L,\mathcal L}_{\theta,c,E,\b,F,\a,G}\hookrightarrow \overline{X}^{\mathcal L}_{\theta, c,E,\b\a,F}& \text{ if }\ \ 0\leq \theta< 1 \\
 &\overline{X}^{\mathcal L,\mathcal L}_{1,c,E,\b,F,\a,G} \hookrightarrow\overline{X}^{\mathcal L}_{1,c,E,\b\|\a\|_{\widetilde{G}(0,t)},F} & \text{ if } \theta= 1,
\end{aligned}$$
and
$$
\begin{aligned}
& \overline{X}^{\mathcal R,\mathcal R}_{\theta,c,E,\b,F,\a,G}\hookrightarrow \overline{X}^{\mathcal R}_{\theta, c,E,\a\b,F} & \text{ if } 0< \theta \leq 1 \\
&\overline{X}^{\mathcal R,\mathcal R}_{0,c,E,\b,F,a,G}  \hookrightarrow\overline{X}^{\mathcal R}_{0,c,E,\b\|\a\|_{\widetilde{G}(t,\infty)},F} & \text{ if } \theta =0.
\end{aligned}
$$

\vspace{2mm}
It is easy to show that the spaces $\mathcal L,\mathcal L$ and $\mathcal R,\mathcal R$ are also related by the following symmetry property.
\begin{lem}\label{symLL}
Let $E$, $F$, $G$ be r.i. spaces, $\a$, $\b$, $c\in SV$ and $0\leq\theta\leq1$. Then
$$(X_0,X_1)^{\mathcal{L,L}}_{\theta,c,E,\b,F,a,G}=(X_1,X_0)^{\mathcal{R,R}}_{1-\theta,\overline{c},E,\overline{b},F,\overline{a},G}$$
where $\overline{a}(t)=a(1/t)$, $\overline{\b}(t)=\b(1/t)$ and $\overline{c}(t)=c(1/t)$,   $t>0$.
\end{lem}

\section{Generalized Holmstedt Type formulae}
For  parameters $0\leq \theta_0<\theta_1\leq 1$ and $\a$, $\b_0$, $\b_1$ slowly varying functions, the couples
 $$(\overline{X}_{\theta_0,\b_0,E_0}, \overline{X}^{\mathcal R}_{\theta_1, \b_1,E_1,\a,F})_{\theta,\b,E}\mand
(\overline{X}^{\mathcal L}_{\theta_0, \b_0,E_0,\a,F}, \overline{X}_{\theta_1,\b_1,E_1})_{\theta,\b,E},$$
are compatible (quasi-) Banach couples (assuming additional conditions on $\b_0$ and $\b_1$ in the extreme cases $\theta_0=0$, $\theta_1=1$, respectively).  In this section we relate the $K$-functionals of these couples with the $K$-functional of the original one, $\overline{X}$, by means of some generalized Holmstedt type formulae. 


\subsection{The $K$-functional of the couple $(\overline{X}_{\theta_0,\b_0,E_0},\overline{X}^{\mathcal R}_{\theta_1, \b_1,E_1,a,F})$, $0\leq \theta_0<\theta_1<1$}\label{HoR}\hspace{2mm}



\begin{thm} \label{5.10.2} 
Let $0<\theta_0<\theta_1<1$. Let $E_0$, $E_1$, $F$ be r.i. spaces and  $\a$, $\b_0$, $\b_1\in SV$ such that  $\|\b_1\|_{\widetilde{E}_1(0,1)}<\infty$.
\vspace{2mm}

\noindent a) Then, for every  $f \in \overline{X}_{\theta_0, \b_0,E_0}+ \overline{X}^{\mathcal R}_{\theta_1,\b_1,E_1,a,F}$ and all $u>0$
 \begin{align*}
 K\big( \rho(u), f;\overline{X}_{\theta_0,\b_0,E_0},\overline{X}^{\mathcal R}_{\theta_1, \b_1,E_1,\a,F}\big ) &\sim
 \| t^{- \theta_0} \b_0(t) K(t,f)\|_{ \widetilde{E}_0(0,u)}\\
&+\rho(u) \| \b_1\|_{\widetilde{E}_1(0,u)} \|t^{- \theta_1} \a(t) K(t,f)\|_{ \widetilde{F}(u,\infty)}\\
&+\rho(u)\Big\| \b_1(t) \|s^{- \theta_1} \a(s) K(s,f)\|_{ \widetilde{F}(t,\infty)}\Big\|_{ \widetilde{E}_1(u,\infty)},
\end{align*}
where 
\begin{equation}\label{rho1}
\rho(u) =u^{\theta_1-\theta_0}\frac{\b_0(u)}{\a(u)\|\b_1\|_{\widetilde{E}_1(0,u)}},\quad u>0.
\end{equation} 
\vspace{1mm}

\noindent b) If  $\|\b_0\|_{\widetilde{E}_0(1,\infty)}<\infty$, then, for every $f \in \overline{X}_{0,\b_0,E_0}+\overline{X}^{\mathcal R}_{\theta_1, \b_1,E_1,\a,F}$ and all $u>0$
 \begin{align*}
 K\big( \rho(u), f;\overline{X}_{0,\b_0,E_0},\overline{X}^{\mathcal R}_{\theta_1, \b_1,E_1,\a,F}\big )&\sim
 \|\b_0(t) K(t,f)\|_{ \widetilde{E}_0(0,u)}\\
&+\rho(u) \| \b_1\|_{\widetilde{E}_1(0,u)} \|t^{- \theta_1} \a(t) K(t,f)\|_{ \widetilde{F}(u,\infty)}\\
&+\rho(u)\Big\| \b_1(t) \|s^{- \theta_1} \a(s) K(s,f)\|_{ \widetilde{F}(t,\infty)}\Big\|_{ \widetilde{E}_1(u,\infty)},
\end{align*}
where $\rho(u) =u^{\theta_1}\frac{\|\b_0\|_{\widetilde{E}_0(u,\infty)}}{\a(u)\|\b_1\|_{\widetilde{E}_1(0,u)}}$, $u>0$
\vspace{1mm}

\noindent c) Moreover, for every  $f \in X_0+\overline{X}^{\mathcal R}_{\theta_1, \b_1,E_1,\a,F}$ and all $u>0$
\vspace{2mm}
 \begin{align}
 K\big( \rho(u), f;X_0,\overline{X}^{\mathcal R}_{\theta_1, \b_1,E_1,\a,F}\big )&\sim
 \rho(u) \| \b_1\|_{\widetilde{E}_1(0,u)} \|t^{- \theta_1} \a(t) K(t,f)\|_{ \widetilde{F}(u,\infty)}\label{x01}\\
&+\rho(u)\Big\| \b_1(t) \|s^{- \theta_1} \a(s) K(s,f)\|_{ \widetilde{F}(t,\infty)}\Big\|_{ \widetilde{E}_1(u,\infty)},\nonumber
\end{align}
where $\rho(u) =u^{\theta_1}\frac{1}{\a(u)\|\b_1\|_{\widetilde{E}_1(0,u)}}$, $u>0$. 
\end{thm}

\begin{proof}
We only proof a), the proofs of b) and c) are similar.
Given $f\in X_0+X_1$ and $u>0$ we consider the  (quasi-) norms
\begin{align*}
(P_0 f)(u) &= \|t^{- \theta_0} \b_0(t) K(t,f) \|_{ \widetilde{E}_0(0,u)},\\
(Q_0 f)(u) &= \| t^{- \theta_0} \b_0(t) K(t,f)\|_{ \widetilde{E}_0(u,\infty)},\\
(P_1 f)(u) &= \Big\| \b_1(t) \|s^{- \theta_1} \a(s) K(s,f)\|_{ \widetilde{F}(t,u)}\Big\|_{ \widetilde{E}_1(0,u)} ,\\
(R_1 f)(u) &= \| \b_1\|_{ \widetilde{E}_1(0,u)} \|t^{- \theta_1} \a(t) K(t,f)\|_{ \widetilde{F}(u,\infty)},\\
(Q_1 f)(u) &= \Big\| \b_1(t)\| s^{- \theta_1} \a(s) K(s,f)\|_{ \widetilde{F}(t,\infty)}\Big\|_{\widetilde{E}_1(u,\infty)},
\end{align*}
and we denote $Y_0= \overline{X}_{\theta_0, \b_0,E_0}$ and $Y_1=\overline{X}^{\mathcal R}_{\theta_1,\b_1,E_1,a,F}$.  With this notation what we pursue to show is the equivalence
\begin{equation}\label{5.17.3.11}
K(\rho(u), f;Y_0,Y_1) \sim (P_{0}f)(u) + \rho(u)[(R_1f)(u)+ (Q_{1}f)(u)],
\end{equation}
for all $f\in Y_0+Y_1$ and $u>0$, where $\rho$ is defined by \eqref{rho1}.

We first prove the upper estimate $\lesssim$ of \eqref{5.17.3.11} for  all $f\in X_0+X_1$ and any positive function $\rho:(0,\infty)\rightarrow(0,\infty)$.

Suppose that $f \in X_0 + X_1$ and fix $u\in(0,\infty)$. We may assume with no loss of generality  that $(P_0f)(u)$, $(R_1 f)(u)$ and $(Q_1f)(u)$ are finite, otherwise the upper estimate of \eqref{5.17.3.11} holds trivially. As usual (see for example~\cite{Bennett-Sharpley} or \cite{EOP}) we choose a decomposition $f = g + h$ 
such that
\begin{equation}\label{E-45}
 \|g\|_{X_0} + u\|h\|_{X_1} \leq 2 K(u,f)
\end{equation}
and  
\begin{equation}\label{45}
K(t,g) \leq 2 K(u,f)\ \ \text{ and }  \ \ \frac{K(t,h)}{t} \leq 2 \frac{K(u,f)}{u}
\end{equation}
for all $t\in(0,\infty)$. So in order to obtain the upper estimate of \eqref{5.17.3.11} it suffices to prove that
$$\|g\|_{Y_0}+\rho(u)\|h\|_{Y_1}\lesssim  (P_{0}f)(u) + \rho(u) [(R_1f(u)+(Q_{1}f)(u)].$$
We start by showing that  $\|g\|_{Y_0}\lesssim  (P_{0}f)(u)$. The triangle inequality and the (quasi-) subadditivity of the $K$- functional establish that $$\|g\|_{Y_0}\leq(P_0g)(u)+(Q_0g)(u)\lesssim (P_0f)(u)+(P_0h)(u)+(Q_0g)(u).$$
Using  (\ref{45}), Lemma \ref{lem1} (i) and (\ref{e1}), we obtain
\begin{align*}
(P_{0}h)(u) &= \big\| t^{- \theta_0} \b_0(t)
K(t,h)\big\|_{\widetilde{E}_0(0,u)} \lesssim \frac{K(u,f)}{u} \|  t^{1-\theta_0} \b_0(t)\|_{\widetilde{E}_0(0,u)}\\& \sim
 u^{- \theta_0} \b_0(u) K(u,f)\lesssim (P_0 f)(u)
\end{align*}
and
\begin{align*}
(Q_{0}g)(u) &= \|  t^{-\theta_0} \b_0(t) K(t,g)\|_{\widetilde{E}_0(u,\infty)} \lesssim
 K(u,f)\| t^{ - \theta_0}\b_0(t)\|_{\widetilde{E}_0(u,\infty)}\\
 &\sim  u^{- \theta_0} \b_0(u) K(u,f)\lesssim
 (P_0f)(u).
 \end{align*}
These give 
$$\|g\|_{Y_0}\lesssim (P_0f)(u)<\infty$$
and therefore $g\in Y_0$. Now we proceed with $\|h\|_{Y_1}$.  Again by  triangle inequality and the (quasi-) subadditivity of the $K$-functional, we have
\begin{align*}
\|h\|_{Y_1}&\leq (P_1h)(u)+(R_1h)(u)+(Q_1h)(u)\\&\lesssim(P_1h)(u)+(R_1f)(u)+(R_1g)(u)+(Q_1f)(u)+(Q_1g)(u).
\end{align*}
 Using (\ref{45}), Lemma \ref{lem1} (i) and \eqref{e3}, we obtain
\begin{align*}
(P_1 h)(u) &= \Big\| \b_1(t) \|s^{- \theta_1} \a(s) K(s,h)\|_{ \widetilde{F}(t,u)}\Big\|_{ \widetilde{E}_1(0,u)}\\
&\lesssim \frac{K(u,f)}{u}\Big\| \b_1(t) \|s^{1- \theta_1} \a(s)\|_{ \widetilde{F}(t,u)}\Big\|_{ \widetilde{E}_1(0,u)}\\
&\lesssim \frac{K(u,f)}{u}\Big\| \b_1(t) \|s^{1- \theta_1} \a(s)\|_{ \widetilde{F}(0,u)}\Big\|_{ \widetilde{E}_1(0,u)}\\
&\sim u^{-\theta_1}\a(u)\|\b_1\|_{\widetilde{E}_1(0,u)}K(u,f)\lesssim (R_1 f)(u)
\end{align*}
and
\begin{align*}
(R_1 g)(u) &= \| \b_1\|_{\widetilde{E}_1(0,u)} \|t^{- \theta_1} \a(t) K(t,g)\|_{ \widetilde{F}(u,\infty)}\\
&\lesssim K(u,f)\|\b_1\|_{\widetilde{E}_1(0,u)}\|t^{-\theta_1}\a(t)\|_{ \widetilde{F}(u,\infty)}\\
&\sim u^{-\theta_1}\a(u)\|\b_1\|_{\widetilde{E}_1(0,u)}K(u,f)\lesssim(R_1 f)(u).
\end{align*}
Similarly, using also \eqref{e6}, we estimate $(Q_1 g)(u)$ from above
 \begin{align*}
(Q_1 g)(u) &= \Big\| \b_1(t) \|s^{- \theta_1} \a(s) K(s,g)\|_{ \widetilde{F}(t,\infty)}\Big\|_{ \widetilde{E}_1(u,\infty)}\\
&\lesssim K(u,f)\Big\| \b_1(t) \|s^{- \theta_1} \a(s)\|_{ \widetilde{F}(t,\infty)}\Big\|_{ \widetilde{E}_1(u,\infty)}\\
&\sim u^{-\theta_1}\b_1(u)\a(u)K(u,f)\lesssim (Q_1 f)(u).
\end{align*}
Thus, 
$$\|h\|_{Y_1}\lesssim  (R_1 f)(u)+ (Q_{1}f)(u)<\infty,$$
and we obtain that $h\in Y_1$. Summing up, we deduce that 
\begin{align*}
K(\rho(u), f;Y_0,Y_1) & \leq \|g\|_{Y_0}+\rho(u)\|h\|_{Y_1}  \\ 
&\lesssim (P_0f)(u)+\rho(u)[(R_1f)(u)+(Q_{1}f)(u)],
\end{align*}
which is the upper estimate of  \eqref{5.17.3.11}.

Let us prove the  lower estimate of \eqref{5.17.3.11}. More precisely,
\begin{equation}\label{ec365.1}
(P_0f)(u)+\rho(u)[(R_1 f)(u)+(Q_1f)(u)]\lesssim K(\rho(u), f;Y_0,Y_1).
\end{equation}
for all $f \in Y_{0} + Y_{1}$, $u>0$ and $\rho$ defined by \eqref{rho1}.

Choose  $f=g+h$  any decomposition of $f$ with $g\in Y_0$ and $h\in Y_1$, and fix again $u>0$. Using the (quasi-) subadditivity of the $K$-functional and the definition of the norm in $Y_0$ and $Y_1$, we have
\begin{align*}
(P_{0}f)(u) & \lesssim (P_{0}g)(u) + (P_{0}h)(u)  \leq  \|g\|_{Y_{0}} + (P_{0}h)(u), \\
(R_1f)(u)&\lesssim (R_1g)(u)+(R_1h)(u)\leq(R_1g)(u)+\|h\|_{Y_1},\\
(Q_{1}f)(u) & \lesssim (Q_{1}g)(u) + (Q_{1}h)(u)  \leq
(Q_{1}g)(u) + \|h\|_{Y_{1}}.
\end{align*}
Then,
\begin{align*}
(P_{0}f)(u) +\rho(u)[(R_1f(u)&+(Q_{1}f)(u)]\\& \lesssim \|g\|_{Y_{0}} + (P_{0}h)(u) +\rho(u)[(R_1g)(u)+(Q_{1}g)(u) + \|h\|_{Y_{1}}].
\end{align*}
Thus, it is enough to verify that $(P_0 h)(u)$, $\rho(u)(R_1g)(u)$ and $\rho(u)(Q_1g)(u)$ are bounded by $\|g\|_{Y_0} +\rho(u)\|h\|_{Y_1}$. We begin with $(P_0h)(u)$. Estimate \eqref{eK2} with $f=h$ and Lemma \ref{lem1} (i) imply that
\begin{align*}
(P_0h)(u)&\lesssim\|h\|_{Y_1}\Big\|t^{\theta_1- \theta_0} \frac{\b_0(t)}{\a(t)\|\b_1\|_{\widetilde{E}_1(0,t)}} \Big\|_{ \widetilde{E}_0(0,u)}
\sim\rho(u)\|h\|_{Y_1}.\nonumber
\end{align*}
Observe that by hypothesis $0<\theta_0<\theta_1<1$ and then $\theta_1-\theta_0>0$.

Similarly, using  \eqref{eK} with $f=g$ and  Lemma \ref{lem1} (i), we have the estimates
\begin{align*}
(R_1g)(u)&\lesssim \| \b_1\|_{ \widetilde{E}_1(0,u)} \Big\|t^{\theta_0- \theta_1} \frac{\a(t)}{\b_0(t)}\Big\|_{ \widetilde{F}(u,\infty)} \|g\|_{Y_0}\\
&\sim u^{\theta_0-\theta_1}\frac{\a(u)}{\b_0(u)} \| \b_1\|_{ \widetilde{E}_1(0,u)}\|g\|_{Y_0}=\frac{1}{\rho(u)}\|g\|_{Y_0}
\end{align*}
and
\begin{align*}
(Q_1g)(u)
&\leq  \|g\|_{Y_0} \bigg\| \b_1(t) \Big\|s^{\theta_0- \theta_1} \frac{\a(s)}{\b_0(s)}\Big\|_{ \widetilde{F}(t,\infty)}\bigg\|_{ \widetilde{E}_1(u,\infty)}\\
&\sim u^{\theta_0-\theta_1}\frac{a(u)\b_1(u)}{\b_0(u)}\|g\|_{Y_0}\lesssim \rho(u)\|g\|_{Y_0}
\end{align*}
where the last equivalence follows from Lemma \ref{lem1} (iii).

Putting together the previous estimates we obtain that
$$(P_{0}f)(u) + \rho(u) [(R_1f)(u)+(Q_{1}f)(u)]\lesssim \|g\|_{Y_0}+\rho(u)\|h\|_{Y_1}.$$
Finally, taking infimum over all possible decomposition of $f=g+h$, with $g\in Y_0$ and $h\in Y_1$, we obtain \eqref{ec365.1} and the proof of a) is finished.
\end{proof}

\subsection{The $K$-functional of the couple $\big(\overline{X}^{\mathcal L}_{\theta_0,\b_0,E_0,\a,F},\overline{X}_{\theta_1, \b_1,E_1}\big)$, $0<\theta_0<\theta_1\leq 1$}\label{HoL}\hspace{2mm}

\vspace{2mm}
Next theorem can be proved  as  Theorem \ref{5.10.2},  although we shall make use of a symmetry argument.

\begin{thm} \label{5.11.2}
Let $0<\theta_0<\theta_1<1$. Let  $E_0$, $E_1$, $F$ be r.i. spaces and  $\a$, $\b_0$, $\b_1\in SV$ such that $\|\b_0\|_{\widetilde{E}_0(1,\infty)}<\infty$.

\vspace{1mm}
\noindent a)  Then, for every $f \in \overline{X}^{\mathcal L}_{\theta_0, \b_0,E_0,\a,F}+ \overline{X}_{\theta_1,\b_1,E_1}$ and all $u>0$
\vspace{1mm}
 \begin{align*}
 K\big( \rho(u), f;\overline{X}^{\mathcal L}_{\theta_0,\b_0,E_0,\a,F},\overline{X}_{\theta_1, \b_1,E_1} \big )& \sim
\Big\|\b_0(t)\| s^{- \theta_0} \a(s) K(s,f)\|_{ \widetilde{F}(0,t)}\Big\|_{\widetilde{E}_0(0,u)}\\
&+\| \b_0\|_{\widetilde{E}_0(u,\infty)} \|t^{- \theta_0} \a(t) K(t,f)\|_{ \widetilde{F}(0,u)}\\
&+\rho(u) \|t^{- \theta_1} \b_1(t) K(t,f)\|_{ \widetilde{E}_1(u,\infty)},
\end{align*}
where
$\rho(u) =u^{\theta_1-\theta_0}\frac{\a(u)\|\b_0\|_{\widetilde{E}_0(u,\infty)}}{\b_1(u)}$, $u>0$.

\vspace{1mm}
\noindent b) If $\|\b_1\|_{\widetilde{E}_1(0,1)}<\infty$, then, for every $f \in \overline{X}^{\mathcal L}_{\theta_0, \b_0,E_0,\a,F}+ \overline{X}_{1,\b_1,E_1}$ and all $u>0$
\vspace{1mm}
 \begin{align*}
 K\big( \rho(u), f; \overline{X}^{\mathcal L}_{\theta_0,\b_0,E_0,\a,F},\overline{X}_{1, \b_1,E_1}\big )&\sim
 \Big\|\b_0(t)\| s^{- \theta_0} \a(s) K(s,f)\|_{ \widetilde{F}(0,t)}\Big\|_{\widetilde{E}_0(0,u)}\\
&+\| \b_0\|_{\widetilde{E}_0(u,\infty)} \|t^{- \theta_0} \a(t) K(t,f)\|_{ \widetilde{F}(0,u)}\\
&+\rho(u) \|t^{-1} \b_1(t) K(t,f)\|_{ \widetilde{E}_1(u,\infty)},
 \end{align*}
where $\rho(u) =u^{1-\theta_0}\frac{\a(u)\|\b_0\|_{\widetilde{E}_0(u,\infty)}}{\|\b_1\|_{\widetilde{E}_1(0,u)}}$, $u>0$.

\vspace{1mm}
\noindent c) Moreover, for every  $f \in \overline{X}^{\mathcal L}_{\theta_0, \b_0,E_0,\a,F}+ X_1$ and all $u>0$
\begin{eqnarray}
 K\big( \rho(u), f; \overline{X}^{\mathcal L}_{\theta_0,\b_0,E_0,\a,F},X_1\big )&\sim
 \Big\|\b_0(t)\| s^{- \theta_0} \a(s) K(s,f)\|_{ \widetilde{F}(0,t)}\Big\|_{\widetilde{E}_0(0,u)}\label{x1L}\\
&+\| \b_0\|_{\widetilde{E}_0(u,\infty)} \|t^{- \theta_0} \a(t) K(t,f)\|_{ \widetilde{F}(0,u)},\nonumber
 \end{eqnarray}
where
$\rho(u) =u^{1-\theta_0}\a(u)\|\b_0\|_{\widetilde{E}_0(u,\infty)}$, $u>0$.
\end{thm}

\begin{proof}
We only prove b), the proofs of a) and c)  are similar.  We consider the slowly varying functions $\overline{\b}_i(t)=\b_i(1/t)$, $i=0,1$, and $\overline{\a}(t)=\a(1/t)$.  Recall that $\|f\|_{\widetilde{E}(\frac1{t},\infty)}=\|f(1/s)\|_{\widetilde{E}(0,t)}$ and hence
$$\frac{1}{\rho(u)}=\Big(\frac{1}{u}\Big)^{1-\theta_0}\frac{\|\overline{\b}_1\|_{\widetilde{E}_1(\frac1u,\infty)}}{\overline{\a}(\frac1u)\|\overline{\b}_0\|_{\widetilde{E}_0(0,\frac1u)}},\quad u>0.$$
By \eqref{eKK} and  Lemma \ref{symLR} we have  that  
\begin{align*}
K\big(\rho(u),& f;\overline{X}^{\mathcal L}_{\theta_0, \b_0,E_0,\a,F},\overline{X}_{1,\b_1,E_1}\big)\\
&=\rho(u)K\Big(\frac1{\rho(u)},f;(X_1,X_0)_{0,\overline{\b}_1,E_1},(X_1,X_0)^{\mathcal R}_{1-\theta_0, \overline{\b}_0,E_0,\overline{\a},F}\Big).
\end{align*}
Now applying Theorem \ref{5.10.2} c)  we obtain the estimate
\begin{align*}
K\big(\rho(u), f;\overline{X}^{\mathcal L}_{\theta_0, \b_0,E_0,\a,F},\overline{X}_{1,\b_1,E_1}\big)&\sim\rho(u)\|\overline{\b}_1(t)K(t,f;X_1,X_0)\|_{\widetilde{E}_1(0,\frac1u)}\\&+\|\overline{\b}_0\|_{\widetilde{E}_0(0,\frac1u)}\|t^{\theta_0-1}\overline{a}(t)K(t,f;X_1,X_0)\|_{\widetilde{F}(\frac1u,\infty)}\\
& + \Big\| \overline{\b}_0(t) \|s^{\theta_0-1} \overline{\a}(s) K(s,f;X_1,X_0)\|_{ \widetilde{F}(t,\infty)}\Big\|_{ \widetilde{E}_0(\frac1u,\infty)}.\end{align*}
Finally the relations $\|f\|_{\widetilde{E}(\frac1{t},\infty)}=\|f(1/s)\|_{\widetilde{E}(0,t)}$ and \eqref{eKK} give the desired equivalence.

\end{proof}

\section{Reiteration formulae for ${\mathcal R}$ and ${\mathcal L}$-spaces.}\label{sereiteration}

The aim of this section is to identify the spaces
$$(\overline{X}_{\theta_0,\b_0,E_0}, \overline{X}^{\mathcal R}_{\theta_1, \b_1,E_1,\a,F})_{\theta,\b,E}\mand(\overline{X}^{\mathcal L}_{\theta_0, \b_0,E_0,\a,F}, \overline{X}_{\theta_1,\b_1,E_1})_{\theta,\b,E}$$
for all possible values of $\theta\in[0,1]$. In that process the lemmas that we collect in the next subsection play a key role.

\subsection{Lemmas}\label{lemmas}\hspace{1mm}

\vspace{2mm}
Next three lemmas  can be found in \cite{FMS-1,FMS-RL1}.
 
\begin{lem}\label{Le51}\cite[Lemma 4.1]{FMS-RL1}
Let $E$ be an r.i. space, $\a$, $\b\in SV$, $0\leq \theta\leq 1$, $0<\alpha<1$  and consider the function $\rho(u)=u^{\alpha}\a(u)$, $u>0$. Then, the equivalence
$$\big\|  \rho(u)^{- \theta}  \b(\rho(u)) K( \rho(u), f)  \big\|_{\widetilde{E}}\sim\big\|  u^{- \theta}  \b(u) K( u, f)  \big\|_{\widetilde{E}}$$
hold for all  $f\in X_0+X_1$, with equivalent constant independent of $f$.
\end{lem}

\begin{lem}\label{key2hardy}\cite[Lemma 2.4]{FMS-1}
Let $\b\in SV$, $\varphi$ a quasi-concave function and $\alpha\in\R$. Then, for any r.i. $E$ and any $t>0$,
\begin{equation}\label{022}
\|s^\alpha\b(s)\varphi(s)\|_{\widetilde{E}(0,t)}
\lesssim \int_0^t s^\alpha\b(s)\varphi(s)\, \frac{ds}{s}
\end{equation}
and
\begin{equation}\label{023}
\|s^\alpha\b(s)\varphi(s)\|_{\widetilde{E}(t,\infty)}
\lesssim \int_t^\infty s^\alpha\b(s)\varphi(s)\, \frac{ds}{s}.
\end{equation}
\end{lem}

\begin{lem}\label{teHardy1}\cite[Lemma 2.5]{FMS-1}
Let $\b\in SV$ and $\alpha>0$. Then, for any r.i. $E$, the inequalities
\begin{equation}\label{eHardy1}
\Bigl\|t^{-\alpha}\b(t)\int_0^t
f(s)\,ds\Bigr\|_{\widetilde{E}}\lesssim
\|t^{1-\alpha}\b(t)f(t)\|_{\widetilde{E}}
\end{equation}
and
\begin{equation}\label{eHardy2}
\Bigl\|t^\alpha\b(t)\int_t^\infty
f(s)\,ds\Bigr\|_{\widetilde{E}}\lesssim
\|t^{1+\alpha}\b(t)f(t)\|_{\widetilde{E}}
\end{equation}
hold for all positive  measurable functions $f$ on $(0,\infty)$.
\end{lem}

Reiteration results of next subsections do not follow from Theorems 4.4, 4.5, 4.7 of \cite{FMS-RL1} and the usual symmetry argument, since the order of the parameters $\theta_0<\theta_1$ is crucial. However, similar proofs can be carried out now using additionally the following lemma from \cite{FMS-4}. 

\begin{lem}\cite[Theorem 3.6]{FMS-4}\label{thmFMS-4}
Let $E$, $F$ be r.i. spaces, $a, b\in SV$ and  $\alpha, \beta \in \R$ with
$\beta>0$. Then, the  equivalence
\begin{equation*}
\Big \|  t^{\beta} \b(t)  \| s^{\alpha}  \a(s)  f(s)\|_{\widetilde{F}(t, \infty)} \Big \|_{\widetilde{E}}  \sim  \|  t^{\alpha + \beta} \a(t)  \b(t)  f(t) \|_{\widetilde{E}}
\end{equation*}
holds for all positive and non-increasing measurable function $f$ on $(0,\infty)$.
\end{lem}


\subsection{The space $( \overline{X}_{\theta_0,\b_0,E_0},\overline{X}^{\mathcal R}_{\theta_1, \b_1,E_1,\a,F})_{\theta,\b,E}$, $0\leq \theta_0<\theta_1<1$ and $0\leq \theta\leq 1$}\label{subR}\hspace{1mm}


\begin{thm}\label{thm5.9}
Let $0<\theta_0<\theta_1<1$ and let $E$, $E_0$, $E_1,\ F$ be r.i. spaces.
Let $\a$, $\b$,  $\b_0$, $\b_1\in SV$ with $\b_1$ satisfying  $\|\b_1\|_{\widetilde{E}_1(0,1)}<\infty$ and consider the function
$$\rho(u) = u^{\theta_1-\theta_0}\frac{\b_0(u)}{\a(u)\|\b_1\|_{\widetilde{E}_1(0,u)} },\quad u>0.$$

\vspace{1mm}
\noindent a) If $0<\theta<1$, then
$$( \overline{X}_{\theta_0,\b_0,E_0},\overline{X}^{\mathcal R}_{\theta_1, \b_1,E_1,\a,F})_{\theta,\b,E}= \overline{X}_{\tilde{\theta},B_\theta,E},$$
where
$$\tilde{\theta}=(1-\theta)\theta_0+\theta\theta_1\quad \mbox{and} \quad B_\theta(u)=\big(\b_0(u)\big)^{1-\theta}\big(\a(u)\|\b_1\|_{\widetilde{E}_1(0,u)}\big)^{\theta}\b(\rho(u)),\ u>0.$$

\vspace{1mm}
\noindent b) If $\theta=0$ and $\|\b\|_{\widetilde{E}(1,\infty)}<\infty$, then
\begin{equation}\label{einter1}
( \overline{X}_{\theta_0,\b_0,E_0},\overline{X}^{\mathcal R}_{\theta_1, \b_1,E_1,\a,F})_{0,\b,E}= \overline{X}^{\mathcal L}_{\theta_0,\b\circ \rho,E,\b_0,E_0}.
\end{equation}

\vspace{1mm}
\noindent c) If $\theta=1$ and $\|\b\|_{\widetilde{E}(0,1)}<\infty$, then
$$
(\overline{X}_{\theta_0,\b_0,E_0},\overline{X}^{\mathcal R}_{\theta_1, \b_1,E_1,\a,F})_{1,\b,E}= \overline{X}^{\mathcal R}_{\theta_1,B_1,E,a,F}\cap\overline{X}^{\mathcal R,\mathcal R}_{\theta_1,\b\circ\rho,E,\b_1,E_1,\a,F},
$$
where $B_1(u)=\|\b_1\|_{\widetilde{E}_1(0,u)}\b(\rho(u))$, $u>0$.
\end{thm}

\begin{proof}
Throughout the proof we use the notation $Y_0=\overline{X}_{\theta_0,\b_0,E_0}$, $Y_1=\overline{X}^{\mathcal R}_{\theta_1, \b_1,E_1,\a,F}$ and $\overline{K}(s,f)=K(s,f;Y_0,Y_1)$, $f\in Y_0+Y_1$, $s>0$.

 We  start with the proof of a). Let $f\in (Y_0,Y_1)_{\theta,\b,E}$. Lemma  \ref{Le51} and Theorem \ref{5.10.2}~a) and the lattice property of $\widetilde{E}$ yield
\begin{align*}
\|f\|_{\overline{Y}_{\theta,\b,E}}&\sim \|\rho(u)^{-\theta}\b(\rho(u))\overline{K}(\rho(u),f)\|_{\widetilde{E}}\\
&\gtrsim \big\|\rho(u)^{-\theta}\b(\rho(u)) \| t^{- \theta_0} \b_0(t) K(t,f)\|_{ \widetilde{E}_0(0,u)}\big\|_{\widetilde{E}}.
\end{align*}
Now using \eqref{e1} and observing that 
\begin{equation}\label{rbtheta}
\rho(u)^{-\theta}\b(\rho(u))=u^{\theta_0-\tilde{\theta}}\frac{\B_{\theta}(u)}{\b_0(u)},\quad u>0,
\end{equation}
one deduces that
$$\|f\|_{\overline{Y}_{\theta,\b,E}}\gtrsim \|\rho(u)^{-\theta}\b(\rho(u)) u^{- \theta_0} \b_0(u) K(u,f)\|_{\widetilde{E}}=\|u^{-\tilde{\theta}}\B_{\theta}(u) K(u,f)\|_{\widetilde{E}}.$$
Thus, the inclusion $\overline{Y}_{\theta,\b,E}\hookrightarrow \overline{X}_{\tilde{\theta},B_\theta,E}$ is proved.

Next we proceed with the reverse inclusion. Let $f\in\overline{X}_{\tilde{\theta},B_\theta,E}$.  Using again Lemma  \ref{Le51}, Theorem \ref{5.10.2} a)  and the triangular inequality,  we have the estimate
\begin{align}
\|f\|_{\overline{Y}_{\theta,\b,E}}&\lesssim \bigg\|\rho(u)^{-\theta}\b(\rho(u)) \| t^{- \theta_0} \b_0(t) K(t,f)\|_{ \widetilde{E}_0(0,u)}\bigg\|_{\widetilde{E}}\label{e46}\\
&+ \bigg\|\rho(u)^{1-\theta}\b(\rho(u)) \| \b_1\|_{\widetilde{E}_1(0,u)}\|t^{- \theta_1} \a(t) K(t,f)\|_{ \widetilde{F}(u,\infty)}\bigg\|_{\widetilde{E}}\nonumber\\
&+\bigg\|\rho(u)^{1-\theta}\b(\rho(u))\Big\| \b_1(t) \|s^{- \theta_1} \a(s) K(s,f)\|_{ \widetilde{F}(t,\infty)}\Big\|_{ \widetilde{E}_1(u,\infty)} \bigg\|_{\widetilde{E}}.\nonumber
\end{align}
We denote the last three expressions by $I_1$, $I_2$ and $I_3$, respectively, and we have to estimate each one by the norm of the function $f$ in $\overline{X}_{\tilde{\theta},B_\theta,E}$. Let us begin with $I_1$. Identity \eqref{rbtheta} implies that
$$I_1=\bigg\|u^{\theta_0-\tilde{\theta}}\frac{\B_{\theta}(u)}{\b_0(u)} \| t^{- \theta_0} \b_0(t) K(t,f)\|_{ \widetilde{E}_0(0,u)}\bigg\|_{\widetilde{E}},$$
where $\theta_0-\tilde{\theta}<0$, so applying \eqref{022} and \eqref{eHardy1} we obtain that $I_1\lesssim \|f\|_{\overline{X}_{\tilde{\theta},B_\theta,E}}$. Indeed,
\begin{align*}
I_1&\lesssim \bigg\| u^{\theta_0-\tilde{\theta}}\frac{\B_{\theta}(u)}{\b_0(u)}\int_0^ut^{- \theta_0} \b_0(t) K(t,f)\ \frac{dt}{t}\bigg\|_{\widetilde{E}}\\
&\lesssim \big\| u^{\theta_0-\tilde{\theta}}\frac{\B_{\theta}(u)}{\b_0(u)}u^{- \theta_0} \b_0(u) K(u,f)\big\|_{\widetilde{E}}\\
&=\|u^{-\tilde{\theta}}B_{\theta}(u)K(u,f)\|_{\widetilde{E}}=\|f\|_{\overline{X}_{\tilde{\theta},B_\theta,E}}.
\end{align*}
Similarly, using that  
\begin{equation}\label{rbtheta2}
 \rho(u)^{1-\theta}\b(\rho(u))=u^{\theta_1-\tilde{\theta}}\frac{\B_{\theta}(u)}{\a(u)\|\b_1\|_{\widetilde{E}_1(0,u)}},\quad u>0,
 \end{equation}
the $\widetilde{L}_1$-bound \eqref{023} and the Hardy type inequality \eqref{eHardy2} ($\theta_1-\tilde{\theta}>0$), we have that
\begin{align}
I_2&= \bigg\|\rho(u)^{1-\theta}\b(\rho(u)) \| \b_1\|_{\widetilde{E}_1(0,u)}\|t^{- \theta_1} \a(t) K(t,f)\|_{ \widetilde{F}(u,\infty)}\bigg\|_{\widetilde{E}}\label{I2}\\
&= \bigg\| u^{\theta_1-\tilde{\theta}}\frac{\B_{\theta}(u)}{\a(u)}\|t^{- \theta_1} \a(t) K(t,f)\|_{ \widetilde{F}(u,\infty)}\bigg\|_{\widetilde{E}}\nonumber\\
&\lesssim \bigg\| u^{\theta_1-\tilde{\theta}}\frac{\B_{\theta}(u)}{\a(u)}\int_u^{\infty}t^{- \theta_1} \a(t) K(t,f)\ \frac{dt}{t}\bigg\|_{\widetilde{E}}\nonumber\\
&\lesssim \big\| u^{\theta_1-\tilde{\theta}}\frac{\B_{\theta}(u)}{\a(u)}u^{- \theta_1} \a(u) K(u,f)\big\|_{\widetilde{E}}\nonumber\\
&=\|u^{-\tilde{\theta}}B_{\theta}(u)K(u,f)\|_{\widetilde{E}}=\|f\|_{\overline{X}_{\tilde{\theta},B_\theta,E}}.\nonumber
\end{align}
Finally, we observe that $I_3$ is bounded by $I_2$. Indeed, using \eqref{rbtheta2} we can identify $I_3$ in the following way
$$
I_3=\bigg\|u^{\theta_1-\tilde{\theta}}\frac{\B_{\theta}(u)}{\a(u)\|\b_1\|_{\widetilde{E}_1(0,u)}}\Big\| \b_1(t) \|s^{- \theta_1} \a(s) K(s,f)\|_{ \widetilde{F}(t,\infty)}\Big\|_{ \widetilde{E}_1(u,\infty)} \bigg\|_{\widetilde{E}}.$$
Since the function $t\rightsquigarrow \|\cdot\|_{ \widetilde{F}(t,\infty)}$, $t\in(u,\infty)$, is a non-increasing function and $\theta_1-\tilde{\theta}>0$,  Lemma \ref{thmFMS-4} gives that
$$
I_3\sim \bigg\|u^{\theta_1-\tilde{\theta}}\frac{\B_{\theta}(u)}{\a(u)\|\b_1\|_{\widetilde{E}_1(0,u)}}\,\b_1(u)\, \|s^{- \theta_1} \a(s) K(s,f)\|_{ \widetilde{F}(u,\infty)} \bigg\|_{\widetilde{E}}\lesssim I_2.$$
Summing up $\|f\|_{\overline{Y}_{\theta,\b,E}}\leq I_1+I_2+I_3\lesssim I_1+I_2\lesssim \|f\|_{\overline{X}_{\tilde{\theta},B_\theta,E}}$. 

The proof of b) follows the same steps. In fact, let $f\in(Y_0,Y_1)_{0,\b,E}$. Lemma  \ref{Le51}, Theorem \ref{5.10.2} a)  and the lattice property of $\widetilde{E}$ yield
$$\|f\|_{\overline{Y}_{0,\b,E}}\sim \|\b(\rho(u))\overline{K}(\rho(u),f)\|_{\widetilde{E}}\gtrsim\Big\|\b(\rho(u))\| t^{- \theta_0} \b_0(t) K(t,f)\|_{ \widetilde{E}_0(0,u)}\Big\|_{\widetilde{E}}.$$
Hence $f\in \overline{X}^{\mathcal L}_{\theta_0,\b\circ \rho,E,\b_0,E_0}$. 

Next we prove the reverse embedding. Let $f\in \overline{X}^{\mathcal L}_{\theta_0,\b\circ \rho,E,\b_0,E_0}$. Arguing as in \eqref{e46} we have
\begin{align*}
\|f\|_{\overline{Y}_{\theta,\b,E}}&\lesssim \bigg\|\b(\rho(u)) \| t^{- \theta_0} \b_0(t) K(t,f)\|_{ \widetilde{E}_0(0,u)}\bigg\|_{\widetilde{E}}\\
&+ \bigg\|\rho(u)\b(\rho(u)) \| \b_1\|_{\widetilde{E}_1(0,u)}\|t^{- \theta_1} \a(t) K(t,f)\|_{ \widetilde{F}(u,\infty)}\bigg\|_{\widetilde{E}}\\
&+\bigg\|\rho(u)\b(\rho(u))\Big\| \b_1(t) \|s^{- \theta_1} \a(s) K(s,f)\|_{ \widetilde{F}(t,\infty)}\Big\|_{ \widetilde{E}_1(u,\infty)} \bigg\|_{\widetilde{E}}:=I_4+I_5+I_6.\end{align*}
Clearly $I_4=\|f\|_{ \overline{X}^{\mathcal L}_{\theta_0,\b\circ \rho,E,\b_0,E_0}}$. In order to estimate $I_5$ by $I_4$ one can argue as in \eqref{I2} with $\theta_1-\theta_0>0$ and use  \eqref{e1} to obtain
\begin{align*}
I_5&= \bigg\|\rho(u)\b(\rho(u)) \| \b_1\|_{\widetilde{E}_1(0,u)}\|t^{- \theta_1} \a(t) K(t,f)\|_{ \widetilde{F}(u,\infty)}\bigg\|_{\widetilde{E}}\\
&= \bigg\|u^{\theta_1-\theta_0}\frac{\b_0(u)}{\a(u)} \b(\rho(u))    \|t^{- \theta_1} \a(t) K(t,f)\|_{ \widetilde{F}(u,\infty)}\bigg\|_{\widetilde{E}}\\
&\lesssim \bigg\| u^{\theta_1-\theta_0}\frac{\b_0(u)}{\a(u)}\b(\rho(u))\int_u^{\infty}t^{- \theta_1} \a(t) K(t,f)\ \frac{dt}{t}\bigg\|_{\widetilde{E}}\\
&\lesssim \big\| u^{\theta_1-\theta_0}\frac{\b_0(u)}{\a(u)}\b(\rho(u))u^{- \theta_1} \a(u) K(u,f)\big\|_{\widetilde{E}}\\
&=\|u^{-\theta_0}\b_0(u)\b(\rho(u))K(u,f)\|_{\widetilde{E}}\\
&\lesssim \bigg\|\b(\rho(u)) \| t^{- \theta_0} \b_0(t) K(t,f)\|_{ \widetilde{E}_0(0,u)}\bigg\|_{\widetilde{E}}=I_4.
\end{align*}
Now we estimate $I_6$ by $I_5$ and then by $I_4$. Using  the definition of $\rho(u)$  and Lemmas  \ref{thmFMS-4} and  \ref{lem1} (iii) we have that
\begin{align*}
I_6&=\bigg\|\rho(u)\b(\rho(u))\Big\| \b_1(t) \|s^{- \theta_1} \a(s) K(s,f)\|_{ \widetilde{F}(t,\infty)}\Big\|_{ \widetilde{E}_1(u,\infty)} \bigg\|_{\widetilde{E}}\\
&=\bigg\|u^{\theta_1-\theta_0}\frac{\b_0(u)}{\a(u)\|\b_1\|_{\widetilde{E}_1(0,u)}} \b(\rho(u)) \Big\| \b_1(t) \|s^{- \theta_1} \a(s) K(s,f)\|_{ \widetilde{F}(t,\infty)}\Big\|_{ \widetilde{E}_1(u,\infty)} \bigg\|_{\widetilde{E}}\\
&\sim\bigg\|u^{\theta_1-\theta_0}\frac{\b_0(u)\b_1(u)}{\a(u)\|\b_1\|_{\widetilde{E}_1(0,u)}} \b(\rho(u))  \|t^{- \theta_1} \a(t) K(t,f)\|_{ \widetilde{F}(u,\infty)} \bigg\|_{\widetilde{E}}\\
&\lesssim\bigg\|u^{\theta_1-\theta_0}\frac{\b_0(u)}{\a(u)} \b(\rho(u))  \|t^{- \theta_1} \a(t) K(t,f)\|_{ \widetilde{F}(u,\infty)} \bigg\|_{\widetilde{E}}=I_5\lesssim I_4.
\end{align*}
Then, $\|f\|_{\overline{Y}_{\theta,\b,E}}\lesssim I_4=\|f\|_{ \overline{X}^{\mathcal L}_{\theta_0,\b\circ \rho,E,\b_0,E_0}}$ and the proof of b) is complete.

Finally, we proceed with the proof of c). Choose $f\in(Y_0,Y_1)_{1,\b,E}$. Lemma \ref{Le51},  Theorem \ref{5.10.2} a) and the lattice property  guarantee that
\begin{align*}
\|f\|_{\overline{Y}_{1,\b,E}}&\sim \|\rho(u)^{-1}\b(\rho(u))\overline{K}(\rho(u),f)\|_{\widetilde{E}}\\
&\gtrsim\Big\|\b(\rho(u))\, \| \b_1\|_{\widetilde{E}_1(0,u)} \|t^{- \theta_1} \a(t) K(t,f)\|_{ \widetilde{F}(u,\infty)}\Big\|_{\widetilde{E}}
\end{align*}
and 
$$
\|f\|_{\overline{Y}_{1,\b,E}}\gtrsim\bigg\|\b(\rho(u))\,\Big\| \b_1(t) \|s^{- \theta_1} \a(s) K(s,f)\|_{ \widetilde{F}(t,\infty)}\Big\|_{ \widetilde{E}_1(u,\infty)}\bigg\|_{\widetilde{E}}
$$
and therefore $f\in\overline{X}^{\mathcal R}_{\theta_1,B_1,E,a,F}\cap\overline{X}^{\mathcal R,\mathcal R}_{\theta_1,\b\circ\rho,E,\b_1,E_1,\a,F}$. Let us prove the reverse embedding.  Again, Lemma \ref{Le51}, Theorem \ref{5.10.2} a) and the triangular inequality give that
\begin{align*}
\|f\|_{\overline{Y}_{1,\b,E}}&\lesssim \bigg\|\rho(u)^{-1}\b(\rho(u)) \| t^{- \theta_0} \b_0(t) K(t,f)\|_{ \widetilde{E}_0(0,u)}\bigg\|_{\widetilde{E}}\\
&+ \bigg\|\b(\rho(u)) \| \b_1\|_{\widetilde{E}_1(0,u)}\|t^{- \theta_1} \a(t) K(t,f)\|_{ \widetilde{F}(u,\infty)}\bigg\|_{\widetilde{E}}\\
&+\bigg\|\b(\rho(u))\Big\| \b_1(t) \|s^{- \theta_1} \a(s) K(s,f)\|_{ \widetilde{F}(t,\infty)}\Big\|_{ \widetilde{E}_1(u,\infty)} \bigg\|_{\widetilde{E}}:=I_7+I_8+I_9.\end{align*}
Since $I_8=\|f\|_{\overline{X}^{\mathcal R}_{\theta_1,B_1,E,a,F}}$ and $I_9=\|f\|_{\overline{X}^{\mathcal R,\mathcal R}_{\theta_1,\b\circ\rho,E,\b_1,E_1,\a,F}}$, it is enough to estimate $I_7$. We proceed as before, applying the definition of $\rho(u)$, and using the equations \eqref{022}, \eqref{eHardy1} and \eqref{e4} it follows that
\begin{align*}
I_7&=\bigg\|\rho(u)^{-1}\b(\rho(u)) \| t^{- \theta_0} \b_0(t) K(t,f)\|_{ \widetilde{E}_0(0,u)}\bigg\|_{\widetilde{E}}\\
&=\bigg\|u^{\theta_0-\theta_1}\frac{\a(u)\|\b_1\|_{\widetilde{E}_1(0,u)}}{\b_0(u)}\,\b(\rho(u)) \| t^{- \theta_0} \b_0(t) K(t,f)\|_{ \widetilde{E}_0(0,u)}\bigg\|_{\widetilde{E}}\\
&\lesssim \bigg\|u^{\theta_0-\theta_1}\frac{\a(u)\|\b_1\|_{\widetilde{E}_1(0,u)}}{\b_0(u)}\,\b(\rho(u)) \int_0^u t^{- \theta_0} \b_0(t) K(t,f)\, \frac{dt}{t}\bigg\|_{\widetilde{E}}\\
&\lesssim \bigg\|u^{-\theta_1}\a(u)\|\b_1\|_{\widetilde{E}_1(0,u)}\b(\rho(u))  K(u,f)\, \bigg\|_{\widetilde{E}}\\
&\lesssim \bigg\|\b(\rho(u)) \| \b_1\|_{\widetilde{E}_1(0,u)}\|t^{- \theta_1} \a(t) K(t,f)\|_{ \widetilde{F}(u,\infty)}\bigg\|_{\widetilde{E}}=I_8.
\end{align*}
Hence
$$\|f\|_{\overline{Y}_{1,\b,E}}\lesssim \max\{\|f\|_{\overline{X}^{\mathcal R}_{\theta_1,B_1,E,a,F}},\, \|f\|_{\overline{X}^{\mathcal R,\mathcal R}_{\theta_1,\b\circ\rho,E,\b_1,E_1,\a,F}}\}$$
and the proof of c) is finished.
\end{proof}

Next we deal with the extreme case $\theta_0=0$.

\begin{thm}\label{thm5.10}
Let $0<\theta_1<1$ and let $E$, $E_0$, $E_1$ and $F$ be r.i. spaces.
Let $\a$, $\b$,  $\b_0$, $\b_1\in SV$ with $\b_0$ and $\b_1$ satisfying  $\|\b_0\|_{\widetilde{E}_0(1,\infty)}<\infty$, $\|\b_1\|_{\widetilde{E}_1(0,1)}<\infty$, respectively, and consider the function
\begin{equation}\label{rho3}
\rho(u) = u^{\theta_1}\frac{\|\b_0\|_{\widetilde{E}_0(u,\infty)}}{\a(u)\|\b_1\|_{\widetilde{E}_1(0,u)} },\quad u>0.
\end{equation}
\vspace{1mm}

\noindent a) If $0<\theta<1$, then
$$( \overline{X}_{0,\b_0,E_0},\overline{X}^{\mathcal R}_{\theta_1, \b_1,E_1,\a,F})_{\theta,\b,E}= \overline{X}_{\tilde{\theta},B_\theta,E}$$
\vspace{1mm}
where
$$\tilde{\theta}=\theta\theta_1\quad \mbox{and} \quad B_\theta(u)=\big(\|\b_0\|_{\widetilde{E}_0(u,\infty)}\big)^{1-\theta}\big(\a(u)\|\b_1\|_{\widetilde{E}_1(0,u)}\big)^{\theta}\b(\rho(u)),\ u>0.$$

\vspace{1mm}
\noindent b) If $\theta=0$ and $\|\b\|_{\widetilde{E}(1,\infty)}<\infty$, then
\begin{equation}\label{einter2}
( \overline{X}_{0,\b_0,E_0},\overline{X}^{\mathcal R}_{\theta_1, \b_1,E_1,\a,F})_{0,\b,E}=  \overline{X}_{0,B_0,E}\cap\overline{X}^{\mathcal L}_{0,\b\circ \rho,E,\b_0,E_0}
\end{equation}
where $B_0(u)=\|\b_0\|_{\widetilde{E}_0(u,\infty)}\b(\rho(u))$, $u>0$.

\vspace{1mm}
\noindent c) If $\theta=1$ and $\|\b\|_{\widetilde{E}(0,1)}<\infty$, then
$$
(\overline{X}_{0,\b_0,E_0},\overline{X}^{\mathcal R}_{\theta_1, \b_1,E_1,\a,F})_{1,\b,E}= \overline{X}^{\mathcal R}_{\theta_1,B_1,E,a,F}\cap\overline{X}^{\mathcal R,\mathcal R}_{\theta_1,\b\circ\rho,E,\b_1,E_1,a,F}
$$
where $B_1(u)=\|\b_1\|_{\widetilde{E}_1(0,u)}\b(\rho(u))$, $u>0$.
\end{thm}

\begin{proof}
The proofs of a) and c) follow the same arguments of the proofs of Theorem \ref{thm5.9} a) and c). The only differences are  the use of Theorem \ref{5.10.2} b) instead of  a) and  the use of the inequality $\b_0(u)\lesssim\|\b_0\|_{\widetilde{E}_0(u,\infty)}$, $u>0$.

Next we prove b). As in Theorem \ref{thm5.9} we use the notation $Y_0=\overline{X}_{0,\b_0,E_0}$, $Y_1=\overline{X}^{\mathcal R}_{\theta_1, \b_1,E_1,\a,F}$ and $\overline{K}(s,f)=K(s,f;Y_0,Y_1)$, $f\in Y_0+Y_1$, $s>0$.
Again, Lemma \ref{Le51} establishes that
$$\|f\|_{\overline{Y}_{0,\b,E}}\sim \|\b(\rho(u))\overline{K}(\rho(u),f)\|_{\widetilde{E}}.$$
Then, to finish the proof it  suffices  to show that $$
 \|\b(\rho(u))\overline{K}(\rho(u),f)\|_{\widetilde{E}}\sim\max \big \{ \|f\|_{\overline{X}_{0,B_0,E}},\|f\|_{\overline{X}^{\mathcal L}_{0,\b\circ \rho,E,\b_0,E_0}}  \big \}.
$$
Theorem \ref{5.10.2} b) and \eqref{e4} guarantee that 
$$\overline{K}(\rho(u),f)\gtrsim \| \b_0(t) K(t,f)\|_{ \widetilde{E}_0(0,u)}$$
and that
\begin{align*}
\overline{K}(\rho(u),f)&\gtrsim \rho(u) \| \b_1\|_{\widetilde{E}_1(0,u)} \|t^{- \theta_1} \a(t) K(t,f)\|_{ \widetilde{F}(u,\infty)}\\
&=u^{\theta_1}\frac{\|\b_0\|_{\widetilde{E}_0(u,\infty)}}{\a(u)\|\b_1\|_{\widetilde{E}_1(0,u)}}\, \| \b_1\|_{\widetilde{E}_1(0,u)} \|t^{- \theta_1} \a(t) K(t,f)\|_{ \widetilde{F}(u,\infty)}\\
&\gtrsim u^{\theta_1}\frac{\|\b_0\|_{\widetilde{E}_0(u,\infty)}}{\a(u)}\, u^{- \theta_1} \a(u) K(u,f)=\|\b_0\|_{\widetilde{E}_0(u,\infty)}K(u,f).
\end{align*}
Hence 
$$\|\b(\rho(u))\overline{K}(\rho(u),f)\|_{\widetilde{E}}\gtrsim
\max  \big \{ \|f\|_{\overline{X}_{0,B_0,E}},\|f\|_{\overline{X}^{\mathcal L}_{0,\b\circ \rho,E,\b_0,E_0}}  \big \}.$$
Now, we proceed to prove the reverse inequality. Use Theorem \ref{thm5.9} b) and the triangular inequality to obtain that
\begin{align*}
\|\b(\rho(u))\, \overline{K}(\rho(u),f)\|_{\widetilde{E}}&\lesssim
 \bigg\|\b(\rho(u)) \| t^{- \theta_0} \b_0(t) K(t,f)\|_{ \widetilde{E}_0(0,u)}\bigg\|_{\widetilde{E}}\\
&+ \bigg\|\rho(u)\b(\rho(u)) \| \b_1\|_{\widetilde{E}_1(0,u)}\|t^{- \theta_1} \a(t) K(t,f)\|_{ \widetilde{F}(u,\infty)}\bigg\|_{\widetilde{E}}\\
&+\bigg\|\rho(u)\b(\rho(u))\Big\| \b_1(t) \|s^{- \theta_1} \a(s) K(s,f)\|_{ \widetilde{F}(t,\infty)}\Big\|_{ \widetilde{E}_1(u,\infty)} \bigg\|_{\widetilde{E}}\\
&:=I_{10}+I_{11}+I_{12}.\end{align*}
The term $I_{10}$ is precisely $\|f\|_{\overline{X}^{\mathcal L}_{0,\b\circ \rho,E,\b_0,E_0}}$. The other  two terms can be estimated by  $\|f\|_{\overline{X}_{0,B_0,E}}$ proceeding as we did in \eqref{I2} and with $I_3$. Indeed,
\begin{align*}
I_{11}&= \bigg\|\rho(u)\b(\rho(u)) \| \b_1\|_{\widetilde{E}_1(0,u)}\|t^{- \theta_1} \a(t) K(t,f)\|_{ \widetilde{F}(u,\infty)}\bigg\|_{\widetilde{E}}\\
&= \bigg\| u^{\theta_1}\frac{B_0(u)}{\a(u)} \|t^{- \theta_1} \a(t) K(t,f)\|_{ \widetilde{F}(u,\infty)}\bigg\|_{\widetilde{E}}\\
&\lesssim \bigg\| u^{\theta_1-\tilde{\theta}}\frac{\B_0(u)}{\a(u)}\int_u^{\infty}t^{- \theta_1} \a(t) K(t,f)\ \frac{dt}{t}\bigg\|_{\widetilde{E}}\\
&\lesssim \big\| u^{\theta_1}\frac{\B_0(u)}{\a(u)}u^{- \theta_1} \a(u) K(u,f)\big\|_{\widetilde{E}}\\
&=\|B_0(u)K(u,f)\|_{\widetilde{E}}=\|f\|_{\overline{X}_{0,B_0,E}}
\end{align*}
and 
\begin{align*}
I_{12}&=\bigg\|\rho(u)\b(\rho(u))\Big\| \b_1(t) \|s^{- \theta_1} \a(s) K(s,f)\|_{ \widetilde{F}(t,\infty)}\Big\|_{ \widetilde{E}_1(u,\infty)} \bigg\|_{\widetilde{E}}\\
&=\bigg\|u^{\theta_1}\frac{\B_0(u)}{\a(u)\|\b_1\|_{\widetilde{E}_1(0,u)}}\Big\| \b_1(t) \|s^{- \theta_1} \a(s) K(s,f)\|_{ \widetilde{F}(t,\infty)}\Big\|_{ \widetilde{E}_1(u,\infty)} \bigg\|_{\widetilde{E}}\\
&\lesssim
\bigg\|u^{\theta_1}\frac{\B_0(u)}{\a(u)\|\b_1\|_{\widetilde{E}_1(0,u)}}\, \b_1(u) \|t^{- \theta_1} \a(t) K(t,f)\|_{ \widetilde{F}(u,\infty)}\bigg\|_{\widetilde{E}}\lesssim I_{11}.
\end{align*}
The proof of b) is complete.
\end{proof}

Our last result of this subsection characterizes the reiteration space when the first space in the couple is $X_0$.

\begin{thm}\label{thm511}
Let $0<\theta_1<1$,  and let $E$,  $E_1$, $F$ be r.i. spaces. Let
 $\a$, $\b$,  $\b_1\in SV$ with $\b_1$  satisfying  $\|\b_1\|_{\widetilde{E}_1(0,1)}<\infty$ and consider the function
$$\rho(u) = u^{\theta_1}\frac{1}{\a(u)\|\b_1\|_{\widetilde{E}_1(0,u)}} ,\quad u>0.$$
Then, the following statements hold:

\vspace{1mm}
\noindent a) If $0<\theta<1$, or $\theta=0$ and  $\|\b\|_{\widetilde{E}(1,\infty)}<\infty$, then
$$(X_0,\overline{X}^{\mathcal R}_{\theta_1, \b_1,E_1,\a,F})_{\theta,\b,E}= \overline{X}_{\tilde{\theta},B_\theta,E},$$
\vspace{1mm}
where
$$\tilde{\theta}=\theta\theta_1 \mand B_\theta(u)=\big(\a(u)\|\b_1\|_{\widetilde{E}_1(0,u)}\big)^{\theta}\b(\rho(u)),\ u>0.$$

\vspace{1mm}
\noindent b) If $\theta=1$ and $\|\b\|_{\widetilde{E}(0,1)}<\infty$, then
$$
(X_0,\overline{X}^{\mathcal R}_{\theta_1, \b_1,E_1,\a,F})_{1,\b,E}=\overline{X}^{\mathcal R}_{\theta_1,B_1,E,a,F}\cap\overline{X}^{\mathcal R,\mathcal R}_{\theta_1,\b\circ\rho,E,\b_1,E_1,a,F},
$$
where $B_1(u)=\|\b_1\|_{\widetilde{E}_1(0,u)}\b(\rho(u))$, $u>0$.
 \end{thm}
 
 \begin{proof}
 Again, as in Theorem \ref{thm5.9} we use the notation $Y_0=X_0$, $Y_1=\overline{X}^{\mathcal R}_{\theta_1, \b_1,E_1,\a,F}$ and $\overline{K}(s,f)=K(s,f;Y_0,Y_1)$, $f\in Y_0+Y_1$, $s>0$.  Lemma \ref{Le51} establishes the equivalence
$$\|f\|_{\overline{Y}_{\theta,\b,E}}\sim \|\rho(u)^{-\theta}\b(\rho(u))\overline{K}(\rho(u),f)\|_{\widetilde{E}}.$$
for $0\leq \theta<1$.  Moreover,  \eqref{x01} and \eqref{e3} imply that
\begin{align*}
\|f\|_{\overline{Y}_{\theta,\b,E}}&\gtrsim \Big\|\rho(u)^{1-\theta}\b(\rho(u)) \| \b_1\|_{\widetilde{E}_1(0,u)}\|t^{- \theta_1} \a(t) K(t,f)\|_{ \widetilde{F}(u,\infty)}\Big\|_{\widetilde{E}}\\
&\gtrsim \bigg\|u^{\theta_1(1-\theta)}\Big(\frac{1}{\a(u)\|\b_1\|_{\widetilde{E}_1(0,u)}}\Big)^{1-\theta}\b(\rho(u)) \| \b_1\|_{\widetilde{E}_1(0,u)}u^{- \theta_1} \a(u) K(u,f)\bigg\|_{\widetilde{E}}\\
&=\|u^{-\theta_1\theta}\B_\theta(u) K(u,f)\|_{\widetilde{E}}=\|f\|_{\tilde{\theta},\B_\theta,E}.
\end{align*}
Hence the inclusion $\overline{Y}_{\theta,\b,E}\hookrightarrow \overline{X}_{\tilde{\theta},\B_\theta,E}$ is proved. The reverse inclusion can be done similarly to the estimate of $I_2$ and $I_3$ in the proof of Theorem \ref{thm5.9} a).

The  case $\theta=1$ can be proved similarly to  Theorem \ref{thm5.9} c) with $I_7=0$.
 \end{proof}

\subsection{The space $( \overline{X}^{\mathcal L}_{\theta_0,\b_0,E_0,\a,F},\overline{X}_{\theta_1, \b_1,E_1})_{\theta,\b,E}$, $0< \theta_0<\theta_1\leq1$ and $0\leq \theta\leq 1$}\label{subL}\hspace{1mm}

\vspace{2mm}
The results of this subsection can be proved using the same ideas we used to prove those of \S \ref{subR}. However, since the proofs are lengthy, we will follow an alternative approach that uses symmetry arguments; besides, some of the proofs will be left to the reader.

\begin{thm}
Let $0<\theta_0<\theta_1<1$ and let $E$, $E_0$, $E_1$ and $F$ be r.i. spaces.
Let $\a$, $\b$,  $\b_0$, $\b_1\in SV$ with $\b_0$ satisfying  $\|\b_0\|_{\widetilde{E}_0(1,\infty)}<\infty$ and consider the function
$$\rho(u) = u^{\theta_1-\theta_0}\frac{\a(u)\|\b_0\|_{\widetilde{E}_0(u,\infty)}}{\b_1(u)},\quad u>0.$$
Then, the following statements hold:

\vspace{1mm}
\noindent a) If $0<\theta<1$, then
$$(\overline{X}^{\mathcal L}_{\theta_0, \b_0,E_0,\a,F}, \overline{X}_{\theta_1,\b_1,E_1})_{\theta,\b,E}= \overline{X}_{\tilde{\theta},B_\theta,E},$$
\vspace{1mm}
where
$$\tilde{\theta}=(1-\theta)\theta_0+\theta\theta_1\quad \mbox{and} \quad B_\theta(u)=\big(\a(u)\|\b_0\|_{\widetilde{E}_0(u,\infty)}\big)^{1-\theta}\big(\b_1(u)\big)^{\theta}\b(\rho(u)),\ u>0.$$

\vspace{1mm}
\noindent b) If $\theta=0$ and $\|\b\|_{\widetilde{E}(1,\infty)}<\infty$, then
$$
(\overline{X}^{\mathcal L}_{\theta_0, \b_0,E_0,\a,F}, \overline{X}_{\theta_1,\b_1,E_1})_{0,\b,E}= \overline{X}^{\mathcal L}_{\theta_0,B_0,E,a,F}\cap\overline{X}^{\mathcal L,\mathcal L}_{\theta_0,\b\circ \rho,E,\b_0,E_0,a,F},$$
where
$B_0(u)=\|\b_0\|_{\widetilde{E}_0(u,\infty)}\b(\rho(u))$, $u>0$.

\vspace{1mm}
\noindent c)If $\theta=1$ and $\|\b\|_{\widetilde{E}(0,1)}<\infty$, then
$$
(\overline{X}^{\mathcal L}_{\theta_0, \b_0,E_0,\a,F}, \overline{X}_{\theta_1,\b_1,E_1})_{1,\b,E}= \overline{X}^{\mathcal R}_{\theta_1,\b\circ\rho,E,\b_1,E_1}.
$$
\end{thm}

\begin{proof}
We express the interpolation spaces  by means of Lemma~\ref{symLR}
\begin{align}\label{e56}
(\overline{X}^{\mathcal L}_{\theta_0, \b_0,E_0,\a,F},\overline{X}_{\theta_1,\b_1,E_1})_{\theta,\b,E}&=
(\overline{X}_{\theta_1,\b_1,E_1},\overline{X}^{\mathcal L}_{\theta_0, \b_0,E_0,\a,F})_{1-\theta,\overline{\b},E}\nonumber\\
&=((X_1,X_0)_{1-\theta_1,\overline{\b}_1,E_1},(X_1,X_0)^{\mathcal R}_{1-\theta_0, \overline{\b}_0,E_0,\overline{a},F})_{1-\theta,\overline{\b},E}.
\end{align}
Here the functions $\overline{a}$, $\overline{b}$ and $\overline{b}_{i}$, $i = 0,1$, have the usual meaning.

Taking $\theta=0$ in \eqref{e56} and applying Theorem \ref{thm5.9} c) we have
$$(\overline{X}^{\mathcal L}_{\theta_0, \b_0,E_0,\a,F},\overline{X}_{\theta_1,\b_1,E_1})_{0,\b,E}=(X_1,X_0)^{\mathcal R}_{1-\theta_0,\B^{\#}_0,E,\overline{a},F}\cap(X_1,X_0)^{\mathcal R,\mathcal R}_{1-\theta_0,\overline{\b}\circ\rho^{\#},E,\overline{\b}_0,E_0,\overline{\a},F},
$$
where $\rho^{\#}(u)=u^{\theta_1-\theta_0}\frac{\overline{\b}_1(u)}{\overline{a}(u)\|\overline{\b}_0\|_{\widetilde{E}_0(0,u)}}$ and $\B^{\#}_0(u)=\|\overline{\b}_0\|_{\widetilde{E}_0(0,u)}\overline{\b}(\rho^{\#}(u))$, $u>0$.
Since $\|\overline{\b}_0\|_{\widetilde{E}_0(0,\frac1u)}=\|\b_0\|_{\widetilde{E}_0(u,\infty)}$ it  yields that 
\begin{align*}
\overline{\B}^{\#}_0(u)&=\B^{\#}_0\Big(\frac{1}{u}\Big)=\|\overline{\b}_0\|_{\widetilde{E}_0(0,\frac{1}{u})}\overline{\b}\bigg(\Big(\frac{1}{u} \Big)^{\theta_1-\theta_0}\frac{\overline{\b}_1(\frac{1}{u})}{\overline{a}(\frac{1}{u})\|\overline{\b}_0\|_{\widetilde{E}_0(0,\frac{1}{u})}}\bigg)\\
&=\|\b_0\|_{\widetilde{E}_0(u,\infty)}\b\Big( u^{\theta_1-\theta_0}\tfrac{\a(u)\|\b_0\|_{\widetilde{E}_0(u,\infty)}}{\b_1(u)}\Big)=\B_0(u)
\end{align*}
and $\overline{\overline{\b}\circ\rho^{\#}}(u)=\b(\rho(u))$, $u>0$  and consequently Lemmas \ref{symLR}  and \ref{symLL} show
$$(\overline{X}^{\mathcal L}_{\theta_0, \b_0,E_0,\a,F}, \overline{X}_{\theta_1,\b_1,E_1})_{0,\b,E}=\overline{X}^{\mathcal L}_{1-\theta_0,\B_0,E,\a,F}\cap\overline{X}^{\mathcal L,\mathcal L}_{1-\theta_0,\b\circ\rho,E,\b_0,E_0,\a,F}.$$ The cases $\theta=1$ and $0<\theta<1$ can be proved similarly.\end{proof}

\begin{thm}\label{thm49}
Let $0<\theta_0<1$ and let $E$, $E_0$, $E_1$ and $F$ be r.i. spaces.
Let $\a$, $\b$,  $\b_0$, $\b_1\in SV$ with $\b_0$ and $\b_1$ satisfying  $\|\b_0\|_{\widetilde{E}_0(1,\infty)}<\infty$, $\|\b_1\|_{\widetilde{E}_1(0,1)}<\infty$, respectively, and consider the function
$$\rho(u) = u^{1-\theta_0}\frac{\a(u)\|\b_0\|_{\widetilde{E}_0(u,\infty)}}{\|\b_1\|_{\widetilde{E}_1(0,u)} },\quad u>0.$$
Then, the following statements hold:

\vspace{1mm}
\noindent a) If $0<\theta<1$, then
$$(\overline{X}^{\mathcal L}_{\theta_0, \b_0,E_0,\a,F},\overline{X}_{1,\b_1,E_1})_{\theta,\b,E}= \overline{X}_{\tilde{\theta},B_\theta,E},$$
\vspace{2mm}
where
$$\tilde{\theta}=(1-\theta)\theta_0+\theta\quad \mbox{and} \quad B_\theta(u)=\big(\a(u)\|\b_0\|_{\widetilde{E}_0(u,\infty)}\big)^{1-\theta}\big(\|\b_1\|_{\widetilde{E}_1(0,u)}\big)^{\theta}\b(\rho(u)),\ u>0.$$

\vspace{1mm}
\noindent b) If $\|\b\|_{\widetilde{E}(1,\infty)}<\infty$, then
$$
(\overline{X}^{\mathcal L}_{\theta_0, \b_0,E_0,\a,F},\overline{X}_{1,\b_1,E_1})_{0,\b,E}= \overline{X}^{\mathcal L}_{\theta_0,B_0,E,\a,F}\cap\overline{X}^{\mathcal L,\mathcal L}_{\theta_0,\b\circ\rho,E,\b_0,E_0,a,F},
$$
where $B_0(u)=\|\b_0\|_{\widetilde{E}_0(u,\infty)}\b(\rho(u))$, $u>0$.

\vspace{1mm}
\noindent c) If $\|\b\|_{\widetilde{E}(0,1)}<\infty$, then
$$(\overline{X}^{\mathcal L}_{\theta_0, \b_0,E_0,\a,F},\overline{X}_{1,\b_1,E_1})_{1,\b,E}= \overline{X}_{1,B_1,E}\cap\overline{X}^{\mathcal R}_{1,\b\circ\rho,E,\b_1,E_1},$$
where $B_1(u)=\|\b_1\|_{\widetilde{E}_1(0,u)}\b(\rho(u))$,  $u>0$.
\end{thm}

\begin{thm}\label{thm6.6}
Let $0<\theta_1<1$ and let $E$, $E_0$, $E_1$ and $F$ be r.i. spaces.
Let $\a$, $\b$,  $\b_0$ with $\b_0$ satisfying  $\|\b_0\|_{\widetilde{E}_0(1,\infty)}<\infty$ and consider the function
$$\rho(u) = u^{1-\theta_0}\a(u)\|\b_0\|_{\widetilde{E}_0(u,\infty)},\quad u>0.$$
Then, the following statements hold:

\vspace{1mm}
\noindent a) If $0<\theta<1$, or $\theta=1$ and $\|\b\|_{\widetilde{E}(0,1)}<\infty$, then
$$(\overline{X}^{\mathcal L}_{\theta_0, \b_0,E_0,\a,F},X_1)_{\theta,\b,E}= \overline{X}_{\tilde{\theta},B_\theta,E},$$
\vspace{1mm}
where
$$\tilde{\theta}=(1-\theta)\theta_0+\theta\quad \mbox{and} \quad B_\theta(u)=\big(\a(u)\|\b_0\|_{\widetilde{E}_0(u,\infty)}\big)^{1-\theta}\b(\rho(u))\ u>0.$$

\vspace{1mm}
\noindent b) If $\|\b\|_{\widetilde{E}(1,\infty)}<\infty$, then
$$
(\overline{X}^{\mathcal L}_{\theta_0, \b_0,E_0,\a,F},X_1)_{0,\b,E}= \overline{X}^{\mathcal L}_{\theta_0,B_0,E,\a,F}\cap\overline{X}^{\mathcal L,\mathcal L}_{\theta_0,\b\circ\rho,E,\b_0,E_0,a,F},
$$
where $B_0(u)=\|\b_0\|_{\widetilde{E}_0(u,\infty)}\b(\rho(u))$, $u>0$.
\end{thm}

\begin{rem}
\textup{
We observe that generalized Holmstedt type formula   \eqref{x01} (or \eqref{x1L}) holds when the space $X_0$ (or $X_1$) is replaced by an intermediate space $\widetilde{X}_0$ of class 0 (or $\widetilde{X}_1$ of class 1, respectively);  see \cite[Chap. 5]{Bennett-Sharpley} for the definition. Consequently, Theorem  \ref{thm6.6} is also true for any intermediate space $\widetilde{X}_0$ of class $0$ and Theorem \ref{thm511} is true for any intermediate space $\widetilde{X}_1$ of class $1$.}
\end{rem}


\section{Applications}\label{applications}

The applications we consider in this section will involve ordered (quasi)- Banach
couples $X = (X_0, X_1)$, in the sense that  $X_1\hookrightarrow X_0$. First we briefly review how our conditions adapt to this simpler setting.

\subsection{Ordered couples}\label{aordered}\hspace{2mm}

\vspace{2mm}

Given a real parameter $0 \leq \theta \leq 1$,  $a, \b, c\in SV(0,1)$ and r.i. spaces $E, F, G$ on $(0,1)$, the spaces  $\overline{X}_{\theta,\b,E}$,
$\overline{X}_{\theta,\b,E,a,F}^{\mathcal L}$ and $\overline{X}_{\theta,c,E,\b,F,a,G}^{\mathcal L,\mathcal L }$  are defined just as in Definitions \ref{defLR} and \ref{defLRR}; the only change being that $\widetilde{E}(0,\infty)$ must be replaced by $\widetilde{E}(0, 1)$, see \cite{FMS-RL1}.
Likewise the spaces  $\overline{X}_{\theta,\b,E,a,F}^{\mathcal R}$  and 
$\overline{X}_{\theta,c,E,\b,F,a,G}^{\mathcal R,\mathcal R}$ are defined as
$$\overline{X}_{\theta,\b,E,a,F}^{\mathcal R}=\Big\{f\in X_0:\Big \|  \b(t) \|   s^{-\theta} a(s) K(s,f) \|_{\widetilde{F}(t,1)}      \Big   \|_{\widetilde{E}(0,1)}<\infty\Big\}$$
and
$$\overline{X}_{\theta,c,E,\b,F,a,G}^{\mathcal R,\mathcal R }=\Big\{f\in X_0: 
\bigg\|c(u)\big\|\b(t) \|s^{-\theta} \a(s) K(s,f) \|_{\widetilde{G}(t,1)}\Big\|_{\widetilde{F}(u,1)}\bigg\|_{\widetilde{E}(0,1)} < \infty\Big\}.$$

Of course, all the results in the paper remain true if we work with an ordered couple  and use as parameters slowly varying functions on $(0,1)$ and r.i. spaces on $(0,1)$. In these cases all assumptions concerning the interval $(1,\infty)$ must be omitted.

It is worth to mention that if the couple is ordered, then the scale $\{\overline{X}_{\theta,\b,E}\}_{0\leq \theta\leq 1}$ is also ordered.

\begin{lem}\label{lemainclusion}\cite[Lemma 5.2]{FMS-RL1}
Let $\overline{X}$ be an ordered (quasi-) Banach couple,  $\b_0,\ \b_1\in SV(0,1)$ and $E_0,\ E_1$ r.i. spaces on $(0,1)$. If $0\leq\theta_0<\theta_1\leq 1$, then
$$\overline{X}_{\theta_1,\b_1,E_1}\hookrightarrow \overline{X}_{\theta_0,\b_0,E_0}.$$
\end{lem}

\vspace{1mm}

\subsection{Grand and small Lebesgue spaces}\label{S_grand}\hspace{2mm}

\vspace{2mm}
Next, we apply our previous results to the grand and small Lebesgue spaces. Following the paper by Fiorenza and Karadzhov  \cite{FK} we give the following definition:
\begin{defn}\label{def_gLp}
Let $(\Omega,\mu)$ be a finite measure space with non-atomic measure $\mu$ and assume that $\mu(\Omega)=1$. Let $1<p<\infty$ and $\alpha>0$. The space $L^{p),\alpha}(\Omega)$ is  the set of all measurable functions $f$ on $(\Omega,\mu)$ such that
\begin{equation}\label{lpa1}
\|f\|_{p),\alpha}=\Big\|\ell^{-\frac{\alpha}{p}}(t)\|f^{*}(s)\|_{L_p(t,1)}\Big\|_{L_\infty(0,1)}<\infty.
\end{equation}
The small Lebesgue space $L^{(p,\alpha}(\Omega)$ is  the set of all measurable functions $f$ on  $(\Omega,\mu)$ such that
\begin{equation}\label{lpa2}
\|f\|_{(p,\alpha}=\Big\|\ell^{\frac{\alpha}{p'}-1}(t)\|f^{*}(s)\|_{L_p(0,t)}\Big\|_{\widetilde{L}_1(0,1)}<\infty
\end{equation}
where $\frac1p+\frac1{p'}=1$.
\end{defn}

The classical grand Lebesgue space $L^{p)}(\Omega):=L^{p),1}(\Omega)$ was introduced by Iwaniec and Sbordone in \cite{IS} in connection with the integrability properties of the Jacobian under minimal hypothesis. The classical small Lebesgue space $L^{(p}(\Omega):=L^{(p,1}(\Omega)$ was introduced by Fiorenza  in \cite{F1}  as dual to the grand Lebesgue spaces; that is $ (L^{(p'})' = L^{p)}.$
Since then many authors have studied relevant properties of these spaces, such as interpolation, boundedness of classical operators, generalized versions, etc. For more information about this spaces and their generalizations see the recent paper \cite{FFG} and the references \cite{AFFGR,AHM,ALM,CO2015,oscar2017,FMS-2,FFGKR}.

We shall also consider ultrasymmetric spaces.
\begin{defn}
Let $1\leq p<\infty$, $\b\in SV$ and $E$ an r.i. space. The ultrasymmetric space $L_{p,\b,E}(\Omega,\mu)$ is  the set of all measurable functions on $(\Omega,\mu)$ such that 
$$\|f\|_{L_{p,\b,E}}=\|t^{1/p}\b(t)f^*(t)\|_{\widetilde{E}}<\infty.$$
\end{defn} 

This class of spaces was introduced and studied by E. Pustylnik \cite{Pust-ultrasymm} and comprises many classical examples as  \textit{Lorentz-Karamata} spaces $L_{p,q;\b}$ (see \cite{GOT,neves}), generalized \textit{Lorentz-Zygmund} spaces  \cite{op}, \textit{Lorentz-Zygmung} spaces spaces $L^{p,q}(\log L)^\alpha$ (see  \cite{BR,Bennett-Sharpley}) and some Orlicz spaces. In case  $E=L_q$ and $b\equiv 1$, we have the classical \textit{Lorentz space} $L_{p,q}$ and the \textit{Lebesgue space} $L_{p}$. 

For convenience we will denote  the functions spaces as $L^{p),\alpha}$, $L^{(p,\alpha}$, $L_p$, etc, dropping the dependence with respect to the domain $(\Omega,\mu)$.

Ultrasymmetric spaces are interpolation spaces for the couple $(L_{1}, L_{\infty})$. Indeed, 
Peetre's well-known formula \cite{Bennett-Sharpley,Peetre}
$$K(t,f;L_1,L_\infty)=\int_0^tf^*(s)\, ds= tf^{**}(t),\quad t>0,$$
and the equivalence $\|t^{1-\theta}\b(t)f^{**}(t)\|_{\widetilde{E}}\sim \|t^{1-\theta}\b(t)f^*(t)\|_{\widetilde{E}}$ for  $0<\theta<1$   (see, e.g. \cite[Lemma 2.16]{CwP}), yield   the equality
\begin{equation}\label{eU}
L_{p,\b,E}=(L_1,L_\infty)_{1-\frac1p,\b,E} 
\end{equation}
 for any r.i. space $E$, 
$\b\in SV(0,1)$ and $1<p<\infty$.

Grand and small Lebesgue spaces are limiting interpolation spaces for the couple $(L_{1}, L_{p})$ and  $(L_{p}, L_{\infty})$, respectively.  Moreover they can also be characterized as $\mathcal{R}$ and  $\mathcal{L}$ spaces, respectively. In fact, it is easy to observe from Definition \ref{def_gLp} that 
$$L^{p),\alpha}=(L_1,L_p)_{1,\ell^{-\frac\alpha{p}}(u),L_\infty}.$$
Then, using the reiteration formula  (4.14) from \cite{GOT} or  \cite[Th. 6.12]{FMS-1}, they can characterized as ${\mathcal R}$ spaces
\begin{equation}\label{eGL}
L^{p),\alpha}=(L_1,(L_1,L_\infty)_{1-\frac1{p},1,L_p})_{1,\ell^{-\frac\alpha{p}}(u),L_\infty}=(L_1,L_\infty)^{\mathcal R}_{1-\frac1p,\ell^{-\frac\alpha{p}}(u),L_\infty,1,L_p}.
\end{equation}
Similarly, using the reiteration formula (3.21) from \cite{GOT}  or \cite[Th. 6.11]{FMS-1}, the small Lebesgue spaces can be seen as  ${\mathcal L}$ spaces
 \begin{align}
 L^{(p,\alpha}&=(L_p,L_\infty)_{0,\ell^{\frac\alpha{p'}-1}(u),L_1}=((L_1,L_\infty)_{1-\frac1{p},1,L_p},L_\infty)_{0,\ell^{\frac\alpha{p'}-1}(u),L_1}\nonumber\\
 &=(L_1,L_\infty)^{\mathcal L}_{1-\frac1p,\ell^{\frac\alpha{p'}-1}(u),L_1,1,L_p}.\label{eSL}
 \end{align}

Since in Corollary 5.3 from \cite{FMS-RL1} we interpolate the grand Lebesgue spaces with the ultrasymmetric spaces included in them, now Theorem \ref{thm5.9} allows us to  obtain the ``dual''  situation.  In other to do that we need some previous considerations.  First of all, if $1<p_0<p_1<\infty$ and $\alpha>0$ then  $(L_{p_0,q;b_0}, L^{p_1),\alpha} )$ is an ordered couple. Indeed, $$L^{p_1),\alpha}\hookrightarrow L^{p_1,\infty}(\log L)^{-\frac{\alpha}{p_1}}\hookrightarrow L_{p_0,q;b_0}.$$
We will also need the following technical lemma.
\begin{lem}\cite[Lemma 6.1]{EOP}
If $\sigma+\frac{1}{q}<0$  with $1\leq q<\infty$ or $q=\infty$ and $\sigma\leq 0$ then 
\begin{equation}\label{einfty1}
\|\ell^{\sigma}(t)\|_{\widetilde{L}_{q}(0,u)}\sim \ell^{\sigma+\frac1{q}}(u), \quad u\in(0,1).
\end{equation}
If $\sigma+\frac{1}{q}>0$  with $1\leq q<\infty$, or $q=\infty$ and $\sigma\geq 0$, then
 \begin{equation}\label{einfty2}
\|\ell^{\sigma}(t)\|_{\widetilde{L}_{q}(u,1)}\sim \ell^{\sigma+\frac1{q}}(u),\quad u\in(0,1/2).\end{equation}
\end{lem}

\vspace{2mm}
Notice also that if $\b(t)\sim\a(t)$ for all $t\in(0,1/2)$, then the monotonicity properties of the $K$-functional and the properties of the slowly varying functions imply that
$$\overline{X}_{\theta,\b,E}=\overline{X}_{\theta,\a,E}.$$
Thus, for any $0<\theta<1$ and any r.i. space $E$,
\begin{equation}\label{equiv}
\overline{X}_{\theta,\|\ell^{\sigma}(t)\|_{\widetilde{L}_{q}(u,1)},E}=\overline{X}_{\theta, \ell^{\sigma+\frac1{q}}(u),E}.
\end{equation}

\begin{cor}\label{713}
Let $E,\ E_0$ be r.i. spaces, $\b\in SV(0,1)$, $1< p_0<p_1<\infty$ and $\beta>0$. Consider the function $\rho(u)=u^{\frac{1}{p_0}-\frac{1}{p_1}}\b_0(u)\ell^{\frac{\beta}{p_1}}(u)$, $u \in (0,1)$.
\begin{itemize}

\vspace{1mm}
\item[a)] If $0<\theta< 1$, then
\begin{equation}\label{e1713}
\big(L_{p_0,\b_0,E_0},L^{p_1),\beta}\big)_{\theta,\b,E}=L_{p,\B_\theta,E},
\end{equation}
where $\frac1{p}=\frac{1-\theta}{p_0}+\frac{\theta}{p_1}$ and $\B_\theta(u)=\b_0^{1-\theta}(u)\ell^{\frac{-\beta\theta}{p_1}}(u)\b(\rho(u))$, $u \in(0,1)$.

\vspace{1mm}
\item[b)] If $\theta=0$, then
\begin{equation}\label{e1714}
\big(L_{p_0,\b_0,E_0},L^{p_1),\beta}\big)_{0,\b,E}=(L_1,L_\infty)^{\mathcal L}_{1-\frac{1}{p_0},\b\circ\rho,E,\b_0,E_0}.
\end{equation}

\vspace{1mm}
\item[c)] If $\theta=1$ and $\|\b\|_{\widetilde{E}(0,1)}<\infty$, then
\begin{align}
\big(L_{p_0,\b_0,E_0},&L^{p_1),\beta}\big)_{1,\b,E}&\label{e1715}\\
&=(L_1,L_\infty)^{\mathcal R}_{1-\frac{1}{p_1},\B_1,E,1,L_{p_1}}\cap(L_1,L_\infty)^{\mathcal R,\mathcal R}_{1-\frac{1}{p_1},\b\circ\rho,E,\ell^{-\frac{\beta}{p_1}}(u),L_\infty,1,L_{p_1}}\nonumber
\end{align}
where $\B_1(u)=\ell^{\frac{-\beta}{p_1}}(u)\b(\rho(u))$, $u \in(0,1)$.

\end{itemize}
\end{cor}
\begin{proof}
Let $0<\theta<1$. Using \eqref{eU}, \eqref{eGL} and Theorem \ref{thm5.9} a) we obtain that
\begin{align*}
\big(L_{p_0,\b_0,E_0},L^{p_1),\beta}\big)_{\theta,\b,E}&=\Big((L_1,L_\infty)_{1-\frac1{p_0},\b_0,E_0},(L_1,L_\infty)^{\mathcal R}_{1-\frac1{p_1},\ell^{-\frac\beta{p_1}}(u),L_\infty,1,L_{p_1}}\Big)_{\theta,\b,E}\\
&=(L_1,L_\infty)_{\widetilde{\theta},B_\theta,E}
\end{align*}
where $\widetilde{\theta}=1-\Big(\frac{1-\theta}{p_0}+\frac{\theta}{p_1}\Big)$ and
$$B_\theta(u)=\b_0^{1-\theta}(u)\|\ell^{-\frac\beta{p_1}}(t)\|^{\theta}_{\widetilde{L}_\infty(0,u)}
\b\bigg(u^{\frac{1}{p_0}-\frac{1}{p_1}}\frac{\b_0(u)}{\|\ell^{-\beta/p_1}(t)\|_{\widetilde{L}_\infty(0,u)}}\bigg),\quad u\in(0,1).$$
We may apply \eqref{eU} to obtain \eqref{e1713}. The proofs of the cases $\theta=0, 1$ can be done similarly.
\end{proof}

In particular, the choice in  \eqref{e1714} of $q=p_0$, $E=L_1$, $E_0=L_{p_0}$ and  functions $\b_0\equiv1$, $\b(u)=\ell^{\frac{\beta}{p_0'}-1}(u)$, $u\in (0,1)$, gives the following result.
\begin{cor}
Let $1<p_0<p_1\leq\infty$
\begin{equation}\label{e1716}
(L_{p_0},L^{p_1),\beta}\big)_{0,\ell^{\frac{\alpha}{p_0'}-1}(u),L_1}=L^{(p_0,\alpha}.
\end{equation}
\end{cor}

Now, we state the interpolation formulae for a couple formed by a grand and a small Lebesgue space. The result recovers Theorem 5.1 from \cite{FFGKR}, and completes it with the extreme cases $\theta=0,\,1$. Moreover, for $0<\theta<1$ this is a special case of \cite[Theorem 6.5]{AFH}.

\begin{thm}\label{thm58}
Let $1<p_0<p_1<\infty$, $1\leq r\leq \infty$ and $\alpha,\beta>0$. 

\vspace{1mm}
\begin{itemize}
\item[a)] If $0<\theta< 1$, then 
$$\big(L^{(p_0,\alpha},L^{p_1),\beta})_{\theta,r}= L^{p,r}(\log L)^{A},$$
where $\frac1p=\frac{1-\theta}{p_0}+\frac\theta{p_1}$ and $A=\frac{\alpha(1-\theta)}{p_0'}-\frac{\beta\theta}{p_1}$.

\vspace{1mm}
\item[b)] If $\theta=0$, then
\begin{align*}
\big(L^{(p_0,\alpha},&L^{p_1),\beta})_{0,r}\\
&=(L_1,L_\infty)^{\mathcal L}_{1-\frac{1}{p_0},\ell^{\frac{\alpha}{p'_0}}(u),L_r,1,L_{p_0}}\cap(L_{p_0},L^{p_1),\beta})^{\mathcal L}_{0,1,L_r,\ell^{\frac{\alpha}{p'_0}-1}(u),L_1}.\nonumber
\end{align*}

\vspace{1mm}
\item[c)] If $\theta=1$ and $\b\in SV(0,1)$ is such that $\|\b\|_{\widetilde{E}(0,1)}<\infty$, then
\begin{align*}
\big(L^{(p_0,\alpha},&L^{p_1),\beta})_{1,\b,E}\\&=
(L_1,L_\infty)^{\mathcal R}_{1-\frac{1}{p_1},\B_1,E,1,L_{p_1}}\cap(L_1,L_\infty)^{\mathcal R,\mathcal R}_{1-\frac{1}{p_1},\b\circ\rho,E,\ell^{-\frac{\beta}{p_1}}(u),L_\infty,1,L_{p_1}}
\end{align*}
where $\rho(u)=u^{\frac{1}{p_0}-\frac{1}{p_1}}\ell^{\frac{\alpha}{p'_0}+\frac{\beta}{p_1}}(u)$ and $\B_1(u)=\ell^{-\frac{\beta}{p_1}}(u)\b(\rho(u))$, $u\in(0,1)$.
 \end{itemize}
\end{thm}

\begin{proof}
Let $0<\theta<1$.  Applying equality \eqref{e1716}, Theorem 5.12 from \cite{FMS-2} and \eqref{einfty2} we obtain the identity
\begin{align*}
\big(L^{(p_0,\alpha},L^{p_1),\beta}\big)_{\theta,r}&=\big((L_{p_0},L^{p_1),\beta})_{0,\ell^{\frac{\alpha}{p_0'}-1}(u),L_1},L^{p_1),\beta}\big)_{\theta,r}\\
&=\big(L_{p_0},L^{p_1),\beta})_{\theta,\ell^{\frac{\alpha(1-\theta)}{p_0'}}(u),L_r}.
\end{align*}
Now use \eqref{e1713} to establish a).
The limiting case $\theta=1$ follows from \eqref{e1715}. 

Finally, assume that $\theta=0$. Then, \eqref{e1716} together with Theorem 5.13 from \cite{FMS-2} and \eqref{einfty2} establish that
\begin{align*}
\big(L^{(p_0,\alpha},L^{p_1),\beta})_{0,r}
&=\big((L_{p_0},L^{p_1),\beta})_{0,\ell^{\frac{\alpha}{p_0'}-1}(u),L_1},L^{p_1),\beta}\big)_{0,r}\\
&=\big(L_{p_0},L^{p_1),\beta})_{0,\ell^{\frac{\alpha}{p'_0}}(u),L_r}\cap\big(L_{p_0},L^{p_1),\beta})^{\mathcal L}_{0,1,L_r,\ell^{\frac{\alpha}{p'_0}-1}(u),L_1}.
\end{align*}
Now, applying \eqref{e1714} we have that
$$\big(L^{(p_0,\alpha},L^{p_1),\beta})_{0,r}=(L_1,L_\infty)^{\mathcal L}_{1-\frac{1}{p_0},\ell^{\frac{\alpha}{p'_0}}(u),L_r,1,L_{p_0}}\cap(L^{(p_0,\alpha},L^{p_1),\beta})^{\mathcal L}_{0,1,L_r,\ell^{\frac{\alpha}{p'_0}-1}(u),L_1}.$$
\end{proof}

\vspace{1mm}
Now, we identify the spaces $\big(L\log L,L^{(p_1,\beta}\big)_{\theta,\b,E}$ and $\big(L_1,L^{(p_1,\beta}\big)_{\theta,\b,E}$ where $1<p_1<\infty$, $\beta>0$, for all possible values of $\theta\in[0,1]$. Remember that
\begin{equation}\label{elog}
(L_1,L_\infty)_{0,1,L_1}=L\log L
\end{equation}
and $L^{(p_1,\beta}\hookrightarrow L\log L\hookrightarrow L_1$ which makes $(L\log L, L^{p_1),\beta})$ and $(L_1,L^{p_1),\beta})$ ordered couples. 

\begin{cor}\label{714}
Let $E$ be an r.i. space, $\b\in SV(0,1)$, $1<p_1<\infty$ and $\beta>0$. Consider the function $\rho(u)=u^{1-\frac{1}{p_1}}\ell^{1+\frac{\beta}{p_1}}(u)$, $u\in(0,1)$.
\begin{itemize}
\item[a)] If $0<\theta< 1$, then
\begin{equation}\label{e1719}
\big(L\log L,L^{p_1),\beta}\big)_{\theta,\b,E}=L_{p,B_\theta,E}
\end{equation}
where $\frac1{p}=1-\theta+\frac{\theta}{p_1}$ and 
$\B_{\theta}(u)=\ell^{1-\theta-\frac{\beta\theta}{p_1}}(u)\b(\rho(u))$, $u\in(0,1)$.

\vspace{1mm}
\item[b)] If $\theta=0$, then
$$\big(L\log L,L^{p_1),\beta}\big)_{0,\b,E}=(L_1,L_\infty)_{0,B_0,E}\cap(L_1,L_\infty)^{\mathcal L}_{0,\b\circ\rho,E,1,L_1}$$
where $B_0(u)=\ell(u)\b(\rho(u))$, $u\in(0,1)$.

\vspace{1mm}
\item[c)] If $\theta=1$ and $\|\b\|_{\widetilde{E}(0,1)}<\infty$, then
$$\big(L\log L,L^{p_1),\beta}\big)_{1,\b,E}=(L_1,L_\infty)^{\mathcal{R}}_{1-\frac{1}{p_1},B_1,E,1,L_{p_1}}\cap(L_1,L_\infty)^{\mathcal R,\mathcal{R}}_{1-\frac{1}{p_1},\b\circ\rho,E,\ell^{-\frac{\beta}{p_1}},1,L_{p_1}}$$
where $B_1(u)=\ell^{-\frac{\beta\theta}{p_1}}(u)\b(\rho(u))$, $u\in(0,1)$.
\end{itemize}
\end{cor}
\begin{proof}
We prove a). By equalities \eqref{eGL}, \eqref{elog} and Theorem \ref{thm5.10}, we have that
\begin{align*}
\big(L\log L,L^{p_1),\beta}\big)_{\theta,\b,E}&=\Big((L_1,L_\infty)_{0,1,L_1},(L_1,L_\infty)^{\mathcal R}_{1-\frac1{p_1},\ell^{-\frac\beta{p_1}}(u),L_\infty,1,L_{p_1}}\Big)_{\theta,\b,E}\\
&=(L_1,L_\infty)_{\tilde{\theta},B_\theta,E}
\end{align*}
where $\tilde{\theta}=\theta\big(1-\frac1{p_1}\big)$ and 
$$B_\theta(u)=\|1\|^{1-\theta}_{\widetilde{L}_1(u,1)}\|\ell^{-\frac{\beta}{p_1}}(t)\|^\theta_{\widetilde{L}_\infty(0,u)}\b\Big(u^{1-\frac{1}{p_1}}\frac{\|1\|_{\widetilde{L}_1(u,1)}}{\|\ell^{-\beta/p_1}(t)\|_{\widetilde{L}_\infty(0,u)}}\Big),\, u\in(0,1).$$
Besides, it follows from equivalences \eqref{einfty1} and \eqref{einfty2}  that
$$B_{\theta}(u) \sim\ell^{1-\theta-\frac{\beta\theta}{p_1}}(u)\b(u^{1-\frac{1}{p_1}}\ell^{1+\frac{\beta}{p_1}}(u)),\quad u\in(1,1/2).$$
Hence, \eqref{equiv} and \eqref{eU}  yield \eqref{e1719}. The remaining cases can be proved similarly.
\end{proof}

Moreover, Theorem \ref{thm511} enables us to  identify the space $\big(L_1,L^{(p_1,\beta}\big)_{\theta,\b,E}$, $0\leq \theta\leq 1$.

\begin{cor}
Let $E$ be an r.i. space, $\b\in SV(0,1)$, $1<p_1<\infty$ and $\beta>0$. Consider the function $\rho(u)=u^{1-\frac{1}{p_1}}\ell^{\frac{\beta}{p_1}}(u)$, $u\in(0,1)$.

\vspace{1mm}
\begin{itemize}
\item[a)] If $0\leq \theta< 1$, then
$$
\big(L_1,L^{p_1),\beta}\big)_{\theta,\b,E}=L_{p,B_\theta,E}
$$
where $\frac1{p}=1-\theta+\frac{\theta}{p_1}$ and 
$\B_{\theta}(u)=\ell^{-\frac{\beta\theta}{p_1}}(u)\b(\rho(u))$, $u\in(0,1)$.

\vspace{1mm}
\item[b)] If $\theta=1$ and $\|\b\|_{\widetilde{E}(0,1)}<\infty$, then
$$\big(L_1,L^{p_1),\beta}\big)_{1,\b,E}=(L_1,L_\infty)^{\mathcal{R}}_{1-\frac{1}{p_1},B_1,E,1,L_{p_1}}\cap(L_1,L_\infty)^{\mathcal R,\mathcal{R}}_{1-\frac{1}{p_1},\b\circ\rho,E,\ell^{-\frac{\beta}{p_1}}(u),1,L_{p_1}}$$
where $B_1(u)=\ell^{-\frac{\beta}{p_1}}(u)\b(\rho(u))$, $u\in(0,1)$.
\end{itemize}
\end{cor}

Our results also allow to interpolate couples formed by small Lebesgue spaces and ultrasymmetric spaces embedded onto them. Observe that  Proposition 5.6 of \cite{CF} and Lemma \ref{lemainclusion} establish that if $1<p_0<p_1<\infty$, $\alpha>0$ and $\beta>1$ then $$L_{p_1,\b_1,E_1}\hookrightarrow L^{p_0}(\log L)^{\frac{\beta\alpha}{p_0'-1}}\hookrightarrow L^{(p_0,\alpha}.$$ The proofs of the following results can be carried out using similar arguments to those of the previous corollaries, so we omit the details.

\begin{cor}
Let $E,\ E_1$ be r.i. spaces, $\b\in SV(0,1)$, $1<p_0<p_1<\infty$ and $\alpha>0$. Consider the function $\rho(u)=
u^{\frac{1}{p_0}-\frac1{p_1}}\ell^{\frac{\alpha}{p'_0}}(u)\b^{-1}_1(u)$, $u>0$. The following statements hold:
\begin{itemize}
\vspace{1mm}
\item[a)] If $0<\theta< 1$, then
$$\big(L^{(p_0,\alpha},L_{p_1,\b_1,E_1}\big)_{\theta,\b,E}=L_{p,\B_\theta,E},$$
where $\frac1{p}=\frac{1-\theta}{p_0}+\frac{\theta}{p_1}$ and $\B_\theta(u)=\ell^{\frac{\alpha(1-\theta)}{p'_0}}(u)\b^{\theta}_1(u)\b(\rho(u))$, $u\in(0,1)$.
\vspace{1mm}
\item[b)] If $\theta=0$, then
$$
\big(L^{(p_0,\alpha},L_{p_1,b_1,E_1}\big)_{0,\b,E}=(L_1,L_\infty)^{\mathcal L}_{1-\frac{1}{p_0},\B_0,E,1,L_{p_0}}\cap(L_1,L_\infty)^{\mathcal L,\mathcal L}_{1-\frac{1}{p_0},\b\circ\rho,E,\ell^{\frac{\alpha}{p'_0}-1}(u),L_1,1,L_{p_0}}
$$
where  $B_0(u)=\ell^{\frac{\alpha}{p'_0}}(u)\b(\rho(u))$, $u\in(0,1)$.
\vspace{1mm}
\item[c)] If $\theta=1$ and $\|\b\|_{\widetilde{E}(0,1)}<\infty$, then
$$\big(L^{(p_0,\alpha},L_{p_1,b_1,E_1}\big)_{1,\b,E}=(L_1,L_\infty)^{\mathcal R}_{1-\frac{1}{p_1},\b\circ\rho,E,\b_1,E_1}.$$
\end{itemize}
\end{cor}

In particular, 
$$\big(L^{(p_0,\alpha},L_{p_1}\big)_{1,\ell^{\frac{-\beta}{p_1}}(u),L_\infty}=L^{p_1),\beta}.$$

We are  also able to interpolate the small Lebesque space $L^{(p_0,\alpha}$ with the Lorentz-Zygmund  spaces $L_{\infty,q}(\log L)^\beta$ and $L^\beta_{\exp}$.  These spaces consist of all measurable functions  $f$ on $(0,1)$  for which the respective norms
\begin{align*}
& \|f\|_{\infty,q,\beta} = \bigg   (  \int_{0}^{1}   \big(\ell^{\beta}(t)  f^{*}(t)     \big)^{q}   \frac{dt}{t}   \Big)^{1/q}, \  \quad \text{if } \beta+\frac1q < 0,\quad (1\leq q<\infty)\\
& \|f\|_{L^\beta_{\exp}}  = \sup_{0 < t < 1}   \ell^{\beta}(t) f^{*}(t), \qquad\qquad \qquad \text{        if } \beta \leq 0,\quad (q=\infty)
\end{align*}
are finite (see \cite{BR}). For simplicity, we jointly denote them by $L_{\infty,q,\beta}$, when $1\leq q\leq \infty$.
They are ultrasymmetric spaces (see \cite[Example 4.4] {FMS-5}) and they can be identified as interpolation spaces between $L_1$ and $L_\infty$ in the following way
\begin{equation}\label{eexp}
(L_1,L_\infty)_{1,\ell^{\beta}(u),L_q}=L_{\infty,q,\beta}. 
\end{equation}
Now, we are in a position to apply  Theorem \ref{thm49}.

\begin{cor}\label{712}
Let $E$ be an r.i. space, $\b\in SV(0,1)$, $1<p_0<\infty$ and $1\leq q_1\leq \infty$. Assume that $\alpha>0$ and $\beta+\frac{1}{q_1}<0$, or $\beta\leq 0$ if $q_1=\infty$, and consider the function
$\rho(u)=u^{\frac{1}{p_0}}\ell^{\frac{\alpha}{p_0'}-(\beta+\frac{1}{q_1})}(u)$, $u\in(0,1)$. The following statements hold:

\vspace{1mm}
\begin{itemize}
\item[a)] If
$0<\theta< 1$, then
$$
\big(L^{(p_0,\alpha},L_{\infty,q_1,\beta}\big)_{\theta,\b,E}=L_{p,\B_\theta,E}$$
where $\frac1{p}=\frac{1-\theta}{p_0}$ and
$\B_{\theta}(u)=\ell^{(1-\theta)\frac{\alpha}{p'_0}+\theta(\beta+\frac{1}{q_1})}(u)\b(\rho(u))$, $u\in(0,1)$.

\vspace{1mm}
\item[b)] If $\theta=0$, then 
$$\big(L^{(p_0,\alpha},L_{\infty,q_1,\beta}\big)_{0,\b,E}=(L_1,L_\infty)^{\mathcal L}_{1-\frac{1}{p_0},B_0,E,1,L_{p_0}}\cap(L_1,L_\infty)^{\mathcal L,\mathcal L}_{1-\frac{1}{p_0},\b\circ\rho,E,\ell^{\frac{\alpha}{p'_0}-1}(u),L_1,1,L_{p_0}}$$
where $B_0(u)=\ell^{\frac{\alpha}{p'_0}}(u)\b(\rho(u))$, $u\in(0,1)$.

\vspace{1mm}\item[c)] If  $\theta=1$ and 
$\|\b\|_{\widetilde{E}(0,1)}<\infty$, then 	
$$\big(L^{(p_0,\alpha},L_{\infty,q_1,\beta}\big)_{1,\b,E}=(L_1,L_\infty)_{1,\B_1,E}\cap(L_1,L_\infty)^{\mathcal R}_{1,\b\circ\rho,E,\ell^{\beta}(u),L_{q_1}}$$
where  $\B_1(u)=\ell^{\beta+\frac{1}{q_1}}(u)\b(\rho(u))$, $u\in(0,1)$.
\end{itemize}
\end{cor}

Finally, we can apply Theorem \ref{thm6.6} to deduce:

\begin{cor}
Let $E$ be an r.i. space, $\b\in SV(0,1)$, $1<p_0<\infty$ and $\alpha>0$. Consider the function
$\rho(u)=u^{\frac{1}{p_0}}\ell^{\frac{\alpha}{p'_0}}(u)$, $u\in(0,1)$.

\vspace{1mm}
\begin{itemize}
\item[a)] If
$0<\theta< 1$, or $\theta=1$ and $\|\b\|_{\widetilde{E}(0,1)}<\infty$, then
$$
\big(L^{(p_0,\alpha},L_{\infty}\big)_{\theta,\b,E}=L_{p,\B_\theta,E}$$
where $\frac1{p}=\frac{1-\theta}{p_0}$ and
$\B_{\theta}(u)=\ell^{\frac{\alpha(1-\theta)}{p'_0}}(u)\b(\rho(u))$, $u\in(0,1)$.

\vspace{1mm}
\item[b)] If $\theta=0$, then
$$\big(L^{(p_0,\alpha},L_{\infty}\big)_{0,\b,E}=(L_1,L_\infty)^{\mathcal L}_{1-\frac{1}{p_0},B_0,E,1,L_{p_0}}\cap(L_1,L_\infty)^{\mathcal L,\mathcal L}_{1-\frac{1}{p_0},\b\circ\rho,E,\ell^{\frac{\alpha}{p'_0}-1}(u),L_1,1,L_{p_0}}$$
where
$\B_0(u)=\ell^{\frac{\alpha}{p'_0}}(u)\b(\rho(u))$, $u\in(0,1)$.
\end{itemize}
\end{cor}

\subsection{Generalized Gamma spaces.}\hspace{1mm}

\vspace{2mm}
Our results can also be applied to the Generalized Gamma spaces with double weight defined in \cite{FFGKR}. 

\begin{defn}
Let $1 \leq p < \infty$, $1 \leq q \leq \infty$ and $w_{1}$, $w_{2}$ two weights on $(0,1)$ satisfying the following conditions:

\vspace{2mm}
\begin{enumerate}
\item[(c1)] There exist $K_{21}>0$ such that $w_2(2t)\leq K_{12}w_2(t)$,  for all $t\in(0,1/2)$. The space $L^p(0,1;w_2)$ is continuously embedded in $L^1(0,1)$.
\item[(c2)] The function $\int_0^t w_2(s)ds$ belongs to $L^{\frac{q}{p}}(0,1;w_1).$
\end{enumerate}
\vspace{2mm}The \textit{Generalized Gamma} space with double weights $G\Gamma(p,q,w_0,w_1)$ is the set of all measurable functions $f$ on $(0,1)$ such that 
$$\|f\|_{G\Gamma}=\bigg(\int_0^1w_1(t)\Big(\int_0^tw_2(s)(f^*(s))^p\ ds\Big)^{\frac{q}{p}}dt\bigg)^{\frac{1}{q}}<\infty$$
\end{defn}

These spaces are a generalization of the $G\Gamma(p,q,w_0):= G\Gamma(p,q,w_0,1)$, introduced in  \cite{FR}, while the  spaces $G\Gamma(p,\infty,w_0,w_1)$ appeared in \cite{GAG}

If we assume that $uw_1(u)$ and $w_2$ are slowly varying functions, we can identify the Generalized Gamma space as an ${\mathcal L}$-space in the following way
$$G\Gamma(p,q,w_1,w_2)=(L_1,L_\infty)^{\mathcal L}_{1-\frac1{p},(uw_1(u))^{\frac{1}{q}},L_q,(w_2(u))^{\frac{1}{p}},L_p} .$$
Thus, we can  apply the results from \S 5.2 to interpolate these spaces with ultrasymmetric spaces.

\begin{cor}
Let $E$, $E_1$ be r.i. spaces and $\b, \b_1, uw_1(u), w_2\in SV(0,1)$. Let $1<p_0<p_1<\infty$, $1\leq q\leq \infty$ and consider the function 
$$\rho(u)=u^{\frac1{p_0}-\frac1{p_1}}\frac{w_2^{\frac1{p_0}}(u)\|(tw_1(t))^{\frac1{q_0}}\|_{\widetilde{L}_{q_0}(u,1)}}{\b_1(u)},\quad u\in(0,1).$$

\begin{itemize}
\item[a)] If $0<\theta< 1$, then
$$\Big(G\Gamma(p_0,q_0,w_1,w_2),L_{p_1,\b_1,E_1}\Big)_{\theta,\b,E}=L_{p,\B_\theta,E}$$
where  $\frac1{p}=\frac{1-\theta}{p_0}+\frac{\theta}{p_1}$ and $$\B_\theta(u)=\Big(w_2(u))^{\frac{1}{p_0}}\|(tw_1(t))^\frac{1}{q_0}\|_{\widetilde{L}_{q_0}(u,1)}\Big)^{1-\theta}\b_1^{\theta}(u)\b(\rho(u)),\   u\in(0,1).$$

\vspace{1mm}
\item[b)] If $\theta=0$ then
\begin{align*}
\Big(G\Gamma(p_0,&q_0,w_1,w_2),L_{p_1,\b_1,E_1}\Big)_{0,\b,E}\\&=(L_1,L_\infty)^{\mathcal L}_{1-\frac{1}{p_0},B_0,E,(w_2(u))^{\frac1{p_0}},L_{p_0}}\cap (L_1,L_\infty)^{\mathcal L,\mathcal L}_{1-\frac{1}{p_0},\b\circ\rho,E,(uw_1(u))^{\frac1{q_0}},L_{q_0},(w_2(u))^{\frac1{p_0}},L_{p_0}}
\end{align*}
where  $\B_0(u)=\Big((w_2(u))^{\frac{1}{p_0}}\|(tw_1(t))^\frac{1}{q_0}\|_{\widetilde{L}_{q_0}(u,1)}\Big)^{1-\theta}\b_1^{\theta}(u)\b(\rho(u))$, $u\in(0,1)$.

\vspace{1mm}
\item[c)] If $\theta=1$ and  $\|\b\|_{\widetilde{E}(0,1)}<\infty$, then
$$\Big(G\Gamma(p_0,q_0,w_1,w_2),L_{p_1,\b_1,E_1}\Big)_{1,\b,E}=(L_1,L_\infty)^{\mathcal R}_{1-\frac{1}{p_1},\b\circ\rho(u),E,\b_1,E_1}.$$
\end{itemize}
\end{cor}

\subsection{$A$ and $B$-type spaces.}\hspace{1mm}

\vspace{2mm}
Finally we consider the $A$ and $B$-type spaces studied by Pustylnik in \cite{pu2}.

\begin{defn}Given $1 <p<\infty$, $\alpha<1$ and $E$ an r.i. space. The space $A_{p,\alpha,E}$ is  the set of all measurable functions $f$ on $(0,1)$ such that
$$ \|f\|_{A_{p,\alpha,E}}=\Big\|\ell^{\alpha-1}(t)\int_t^1s^{\frac{1}{p}}f^{**}(s)\,\frac{ds}s\Big\|_
{\widetilde{E}}<\infty$$
assumed that the function $(1+u)^{\alpha-1}$ belongs to $E$ (i.e. $\|\ell^{\alpha-1}(t)\|_{\widetilde{E}(0,1)}<\infty$.)
The space $B_{p,\b,E}$ is the set of all measurable functions $f$ on $(0,1)$ such that 
$$
\|f\|_{B_{p,\alpha,E}}=\Big\|\sup_{0<s<t}s^{\frac{1}{p}}\ell^{\alpha-1}(s) f^{**}(s)\Big\|_{\widetilde{E}}<\infty.$$
\end{defn}

The space of $B$-type when $\alpha=0$  first appeared in \cite{CwP}. General versions of these spaces were studied in \cite{PS1}. The main advantage of the $A$ and $B$-type spaces is their optimality in the weak interpolation \cite{pu2,PS2}.

The $A$ and $B$-type spaces can be seen as $\mathcal{R}$ and $\mathcal{L}$ spaces, respectively. Indeed,
$$A_{p,\alpha,E}=(L_1,L_\infty)^{\mathcal{R}}_{1-\frac1p,\ell^{\alpha-1}(t),E,1,L_1}\quad B_{p,\alpha,E}=(L_1,L_\infty)^{\mathcal{L}}_{1-\frac1p,1,E,\ell^{\alpha-1}(t),L_\infty}.$$
Then, we can apply the results from \S \ref{sereiteration} to obtain the following interpolation formulae.

\begin{cor}
Let $E,\ E_0,\ E_1$ be r.i. spaces, $\b,\b_0\in SV(0,1)$, $1< p_0<p_1< \infty$, $\beta<1$ and assume that $(1+u)^{\beta-1}$ belongs to $E_1$. Consider the function 
$\rho(u)=u^{\frac{1}{p_0}-\frac{1}{p_1}}\b_0(u)\|\ell^{\beta-1}(t)\|^{-1}_{\widetilde{E}_1(0,u)}$, $u \in (0,1)$.
\begin{itemize}

\vspace{1mm}
\item[a)] If $0<\theta< 1$, then
\begin{equation*}
\big(L_{p_0,\b_0,E_0},A_{p_1,\beta,E_1}\big)_{\theta,\b,E}=L_{p,\B_\theta,E},
\end{equation*}
where $\frac1{p}=\frac{1-\theta}{p_0}+\frac{\theta}{p_1}$ and $\B_\theta(u)=\b_0^{1-\theta}(u)\|\ell^{\beta-1}(t)\|^{\theta}_{\widetilde{E}_1(0,u)}\b(\rho(u))$, $u \in(0,1)$.

\vspace{1mm}
\item[b)] If $\theta=0$, then
$$
\big(L_{p_0,\b_0,E_0},A_{p_1,\beta,E_1}\big)_{0,\b,E}=(L_1,L_\infty)^{\mathcal L}_{1-\frac{1}{p_0},\b\circ\rho,E,\b_0,E_0}.$$

\vspace{1mm}
\item[c)] If $\theta=1$ and  $\|\b\|_{\widetilde{E}(0,1)}<\infty$, then
\begin{align*}
\big(L_{p_0,\b_0,E_0},&A_{p_1,\beta,E_1}\big)_{1,\b,E}\\
&=(L_1,L_\infty)^{\mathcal R}_{1-\frac{1}{p_1},\B_1,E,1,L_{1}}\cap(L_1,L_\infty)^{\mathcal R,\mathcal R}_{1-\frac{1}{p_1},\b\circ\rho,E,\ell^{\beta-1}(u),E_1,1,L_1}\nonumber
\end{align*}
where $\B_1(u)=\|\ell^{\beta-1}(t)\|_{\widetilde{E}_1(0,u)}\b(\rho(u))$, $u \in(0,1)$.
\end{itemize}
\end{cor}

\begin{cor}
Let $E,\, E_0,\, E_1$ be r.i. spaces, $\b,\ \b_1\in SV(0,1)$, $1<p_0<p_1<\infty$ and $\alpha<1$. Consider the function $\rho(u)=u^{\frac{1}{p_0}-\frac{1}{p_1}}\ell^{\alpha-1}(u)\varphi_{E_0}(\ell(u))\b^{-1}_1(u)$, $u \in (0,1)$.

\vspace{1mm}
\begin{itemize}
\item[a)] If $0<\theta< 1$, then
$$\big(B_{p_0,\alpha,E_0},L_{p_1,\b_1,E_1}\big)_{\theta,\b,E}=L_{p,\B_\theta,E},$$
where $\frac1{p}=\frac{1-\theta}{p_0}+\frac{\theta}{p_1}$ and $\B_\theta(u)=\big(\ell^{\alpha-1}(u)\varphi_{E_0}(\ell(u))\big)^{1-\theta}\b^\theta_1(u)\b(\rho(u))$, $u \in(0,1)$.

\vspace{1mm}
\item[b)] If $\theta=0$, then
\begin{align*}
\big(B_{p_0,\alpha,E_0},&L_{p_1,\b_1,E_1}\big)_{0,\b,E}\\
&=(L_1,L_\infty)^{\mathcal L}_{1-\frac{1}{p_0},\B_0,E,\ell^{\alpha-1}(u),L_\infty}\cap(L_1,L_\infty)^{\mathcal L,\mathcal L}_{1-\frac{1}{p_0},\b\circ\rho,E,1,E_0,\ell^{\alpha-1}(u),L_\infty}
\end{align*}
where $B_0(u)=\varphi^{1-\theta}_{E_0}(\ell(u))\b(\rho(u))$, $u\in(0,1)$.

\vspace{1mm}
\item[c)] If $\theta=1$ and  $\|\b\|_{\widetilde{E}(0,1)}<\infty$, then
$$
\big(B_{p_0,\alpha,E_0},L_{p_1,\b_1,E_1}\big)_{1,\b,E}=(L_1,L_\infty)^{\mathcal R}_{1-\frac{1}{p_1},\b\circ\rho,E,\b_1,E_1}.$$
\end{itemize}
\end{cor}

In particular, the $B$-type spaces (or $A$-type) can be seen as limiting interpolation spaces between the ultrasymmetric spaces and $A$-type spaces (or $B$-type, respectively).
\begin{cor}
Let $E_0,\ E_1$ be r.i. spaces, $1<p_0<p_1<\infty$ and $\alpha,\beta<1$.  Then
$$(L_{p_0,\ell^{\alpha-1}(t),L_\infty},A_{p_1,\beta,E_1})_{0,1,E_0}=B_{p_0,\alpha,E_0}$$
and 
$$(B_{p_0,\alpha,E_0},L_{p_1,1,L_1})_{1,\ell^{\beta-1}(t),E_1}=A_{p_1,\beta,E_1}.$$
\end{cor}

Finally, ultrasymmetric spaces are interpolation spaces between $A$ and $B$-type spaces. The proof of the following corollary is similar to the proof of Theorem  \ref{thm58}.

\begin{cor}
Let $E,\, E_0,\, E_1$ be r.i. spaces,  $1< p_0<p_1< \infty$, $\alpha,\beta<1$ and assume that $(1+u)^{\beta-1}$ belongs to $E_1$. Consider the function $$\rho(u)=u^{\frac{1}{p_0}-\frac{1}{p_1}}\frac{\ell^{\alpha-1}(u)\varphi_{E_0}(\ell(u))}{\|\ell^{\beta-1}(t)\|_{\widetilde{E}_1(0,u)}},\qquad u\in(0,1).$$

\vspace{2mm}
a)  If $0<\theta< 1$, then
\begin{equation*}
\big(B_{p_0,\alpha,E_0},A_{p_1,\beta,E_1}\big)_{\theta,\b,E}=L_{p,\B_\theta,E},
\end{equation*}
where $\frac1{p}=\frac{1-\theta}{p_0}+\frac{\theta}{p_1}$ and $\B_\theta(u)=\big(\ell^{(\alpha-1)}(u)\varphi_{E_0}(\ell(u))\big)^{1-\theta}\|\ell^{\beta-1}(t)\|^{\theta}_{\widetilde{E}_1(0,u)}\b(\rho(u))$, $u\in(0,1)$.

\vspace{2mm}
b) If $\theta=0$, then
\begin{align*}
\big(&B_{p_0,\alpha,E_0},A_{p_1,\beta,E_1}\big)_{0,\b,E}\\ &=(L_1,L_\infty)^{\mathcal L}_{1-\frac{1}{p_0},\varphi_{E_0}(\ell(u))\b(\rho(u)),E,\ell^{\alpha-1}(u),L_\infty}\cap(L_{p_0,\ell^{\alpha-1}(u),L_\infty},A_{p_1,\beta,E_1})^{\mathcal L}_{0,\b(t\varphi_0(\ell(t)),E,1,E_0}.
\end{align*}

\vspace{2mm}
c) If $\theta=1$ and $b\in SV(0,1)$ is such that $\|\b\|_{\widetilde{E}(0,1)}<\infty$, then
\begin{align*}
\big(B_{p_0,\alpha,E_0},&A_{p_1,\beta,E_1}\big)_{1,\b,E}\\
&=(L_1,L_\infty)^{\mathcal R}_{1-\frac{1}{p_1},\B_1,E,1,L_1}\cap(L_1,L_\infty)^{\mathcal R,\mathcal R}_{1-\frac{1}{p_1},\b\circ\rho,E,\ell^{\beta-1}(u),E_1,1,L_1}\nonumber
\end{align*}
where $\B_1(u)=\|\ell^{\beta-1}(t)\|_{\widetilde{E}_1(0,u)}\b(\rho(u))$, $u \in(0,1)$.
\end{cor}

\
 
{\small \hspace{-5mm}{\textbf{Acknowledgments.}}
The authors have been partially supported by grant MTM2017-84058-P (AEI/FEDER, UE).
The second author also thanks the Isaac Newton Institute for Mathematical Sciences, Cambridge, for support and hospitality during the programme \emph{Approximation, Sampling and Compression in Data Science} where the work on this paper was undertaken; this work was supported by EPSRC grant no. EP/R014604/1. Finally, the second author thanks Óscar Domínguez for useful conversations at the early stages of this work, and for pointing out the reference \cite{FFGKR}. }

\end{document}